\documentclass[10pt, a4paper]{article}

\usepackage{times}
\usepackage{amsmath}
\usepackage{amsfonts}
\usepackage{amssymb}
\usepackage{graphicx}
\usepackage{authblk}
\usepackage{babel}
\usepackage{amsthm}
\usepackage[font=small,labelfont=bf]{caption}
\DeclareCaptionFont{ninept}{\fontsize{9pt}{11pt}\selectfont}
\usepackage{subcaption}
\usepackage{physics}
\usepackage{listings}
\usepackage{float}
\usepackage{mathrsfs}
\usepackage{mathtools}
\usepackage{parskip}
\usepackage{microtype}
\usepackage{setspace}
\usepackage{booktabs}
\usepackage{enumitem}
\usepackage{xcolor}
\usepackage{etoolbox}
\apptocmd{\thebibliography}
  {\setlength{\parskip}{0pt}
   \setlength{\itemsep}{0pt}
   \setlength{\parsep}{0pt}
   \setlength{\topsep}{0pt}
   \setlength{\partopsep}{0pt}}
  {}{}
\usepackage{ragged2e}
\AtBeginDocument{\justifying}
\usepackage{titlesec}
\usepackage{fancyhdr}
\usepackage[a4paper,
top    = 1.5cm,
bottom = 1.5cm,
left   = 1.5cm,
right  = 1.5cm]{geometry}

\title{\centering\bfseries\fontsize{16pt}{19.2pt}\selectfont A Hamilton–Jacobi Framework in a Field–Road System with Unidirectional Advection under Wentzell-Type Boundary Condition}
\author[*,1]{\centering\fontsize{12pt}{14.4pt}\selectfont Xinye Xiao}
\affil[1]{\centering\fontsize{12pt}{14.4pt}\selectfont School of Mathematics and Physics, China University of Geosciences, Wuhan, P. R. China}
\affil[1]{\centering\fontsize{12pt}{14.4pt}\selectfont Email: 20221003491@cug.edu.cn}
\author[2]{\centering\fontsize{12pt}{14.4pt}\selectfont Haomin Huang}
\affil[2]{\centering\fontsize{12pt}{14.4pt}\selectfont School of Mathematics and Physics, China University of Geosciences, Wuhan, P. R. China}
\affil[2]{\centering\fontsize{12pt}{14.4pt}\selectfont Email: huanghaomin@cug.edu.cn}

\date{}

\theoremstyle{plain}
\newtheorem{proposition}{Proposition}
\newtheorem{definition}{Definition}
\newtheorem{theorem}{Theorem}
\newtheorem{lemma}{Lemma}
\newtheorem{remark}{Remark}

\newcommand{\R}{\mathbb{R}}
\newcommand{\Rplus}{\mathbb{R}_+}

\newcommand{\Chi}{\chi}
\newcommand{\T}{\mathcal{T}}
\newcommand{\B}{\mathcal{B}}

\renewenvironment{abstract}
  {\begin{center}
     \bfseries\fontsize{11pt}{13.2pt}\selectfont
     ABSTRACT
   \end{center}
   \vspace{-12pt}
   \vspace{24pt}
   \fontsize{10pt}{12pt}\selectfont
   \noindent\ignorespaces}
   {\par}

\newcommand{\keywords}[1]{
  \vspace{0pt}
  \noindent
  \textbf{Keywords:} #1
  \par
}

\begin{document}

\maketitle

\thispagestyle{empty}

\begin{abstract}
This paper develops a comprehensive Hamilton-Jacobi framework to analyze asymptotic propagation dynamics in a field-road system featuring unidirectional advection and Wentzell-type boundary conditions. We rigorously derive a Hamilton-Jacobi variational inequality as the singular limit of a reaction-diffusion system in the upper half-plane, where the road is modeled as a degenerate one-dimensional medium with enhanced diffusion and tangential drift. By synthesizing viscosity solution theory, optimal control formulation, and variational analysis, we establish the existence, uniqueness, and explicit variational representation of the viscosity solution. The solution is characterized by a fundamental solution constructed via optimal paths, revealing a critical transition in propagation behavior governed by a geometrically derived curve that separates rectilinear and road-assisted regimes. Our framework extends to non-order-preserving systems where classical comparison methods fail, and we provide a detailed asymptotic derivation of the Wentzell boundary condition from flux continuity principles. Furthermore, we generalize the approach to conical domains with intersecting roads, demonstrating the robustness of our variational methodology. Numerical simulations illustrate how advection and diffusion parameters shape the invaded region, highlighting the interplay between field and road dynamics in determining propagation patterns.
\end{abstract}


\section{INTRODUCTION}
This paper focuses on the analysis of a Hamilton-Jacobi variational inequality with drift under Wentzell-type boundary conditions. To place the variational inequality (1.1) in its modeling context, we briefly recall the physical geometry and the limiting procedure that leads to the Wentzell-type boundary condition. We begin with a two-layer field–road model in which the linear corridor (the "road") is represented by a strip of finite thickness \(\delta>0\) centered at \(y=0\) and endowed with enhanced longitudinal diffusivity and a possibly different advective velocity. In the full parabolic system, the concentrations in the strip and in the upper half-plane are coupled by the continuity of flux across the microscopic interfaces. Performing an asymptotic homogenization / singular limit as \(\delta \to 0\) and applying a WKB ansatz yield, after appropriate rescaling, an effective description on the half-plane in which the road enters as an enhanced boundary operator. The Wentzell-type boundary condition emerges naturally as the limiting flux continuity relation (detailed derivation in Appendix A). This presentation justifies treating the road as a degenerate one-dimensional medium coupled to the field by a higher-order boundary operator and motivates the variational–Hamilton–Jacobi formulation below. Subsequently, the following system is obtained:
\begin{equation}
\begin{cases}
\min\big\{v_t + |\nabla v|^2 + c v_x + 1,\ v\big\} = 0, & \quad x \in \R, \, y > 0, \, t \in \R_+,\\
a v_x^2 - v_y + b v_x = 0, & \quad x \in \R, \, y = 0, \, t \in \R_+,\\
v(x,y,0) = 0, & \quad (x,y) = (0,0),\\
v(x,y,t) \to \infty \quad \text{as } t \to 0+, & \quad (x,y) \in \overline{\Rplus^2} \setminus \{(0,0)\}.
\end{cases}
\tag{1.1}
\end{equation}

This variational inequality arises naturally in the \emph{geometric optics} approximation of reaction-diffusion systems with unidirectional advection. It was rigorously derived by Li and Wang \cite{ref1} as the singular limit of a field-road paradigm, where the solution $u$ of the Fisher-KPP equation 
\[
\begin{cases}
u_t - \Delta u + c u_x = u(1 - u), & \quad x \in \R, \, y > 0, \, t \in \R_+,\\
a u_{xx} - b u_x + u_y = 0, & \quad x \in \R, \, y = 0, \, t \in \R_+,\\
u(x, y, 0) = u_0(x, y), & \quad x \in \R, \, y \geq 0.
\end{cases}
\tag{1.2}
\]

admits a WKB ansatz $u^\varepsilon = \exp(-v^\varepsilon/\varepsilon)$. Under hyperbolic scaling, the phase function $v^\varepsilon$ converges to the viscosity solution of (1.1) as $\varepsilon \to 0$ \cite{ref2,ref3}.

The field-road paradigm models biological invasions in fragmented landscapes where linear infrastructures (e.g., rivers, hedgerows, or anthropogenic corridors) profoundly alter dispersal dynamics. Consider an invasive species propagating through a two-dimensional habitat ($y>0$) bisected by a unidirectional thoroughfare ($y=0$). The advection term $b u_x$ captures directed movement along the corridor (e.g., wind-driven seed dispersal or aquatic transport in riverine systems), while the degenerate diffusion coefficient $a$ quantifies enhanced motility within the corridor—typically orders of magnitude higher than in the surrounding terrain. This dichotomy arises from empirical observations: linear features facilitate rapid long-distance dispersal via hydrodynamic forces or human-mediated transport, whereas the matrix habitat imposes resistance due to vegetation density or resource scarcity. The Wentzell condition $au_{xx} - bu_x + u_y = 0$ embodies flux continuity across the interface, reconciling Fickian diffusion in the bulk with anomalous superdiffusion along the corridor. Critical transitions in propagation speeds delineate regimes where invasion fronts leverage the corridor as a wave-guide versus scenarios where bulk diffusion dominates.

The coefficient $a > 0$ represents the transport capacity of the degenerate road. The parameter $c$ represents the advective velocity in the bulk field, where a positive value of $c$ indicates advection in the $(+x)$-direction. The parameter $b$ appears in the boundary operator and characterizes a tangential drift along the road. There is no inherent mathematical requirement for $b$ and $c$ to share the same sign; physically and biologically, the direction of the road drift may either align with or oppose the field advection. Nevertheless, the relative signs and magnitudes of $b$ and $c$ play a significant role in determining the propagation dynamics. These parameters jointly influence the critical boundary and the variational formula, leading to two distinct regimes: (i) when $b$ and $c$ have the same sign, their effects are additive in the downstream direction, enhancing propagation; and (ii) when their signs are opposite, competing transport mechanisms arise, resulting in asymmetric front profiles and potentially reduced downstream propagation speed. For the sake of consistency and clarity in the subsequent analysis, it is assumed that both b and c are positive throughout the discussion.

We establish that the viscosity solution of (1.1) admits an explicit variational representation:
\begin{equation}
v(x,y,t) = \max\left\{0, \varphi^*(x,y,t) - t\right\},
\tag{1.3}
\end{equation}

where the fundamental solution $\varphi^*$ is given by
\[
\varphi^*(x,y,t) := \min_{s \geq 0} \left\{ \frac{\left(-x + b\,s + c\,t\right)^2}{4(t+as)} + \frac{(y+s)^2}{4t} \right\}.
\]

Key properties of $\varphi^*$ include:
\begin{itemize}
    \item Homogeneity: $\varphi^*(tx,ty,t) = t\varphi^*(x,y,1)$
    \item $\mathcal{C}^1$-regularity across transition boundaries (Proposition 2)
    \item Strict convexity of the level set $\Omega = \{ (x,y) : \varphi^*(x,y,1) < 1 \}$ (Appendix B)
\end{itemize}

\begin{figure}[H]
    \centering
    \includegraphics[width=0.3\textwidth]{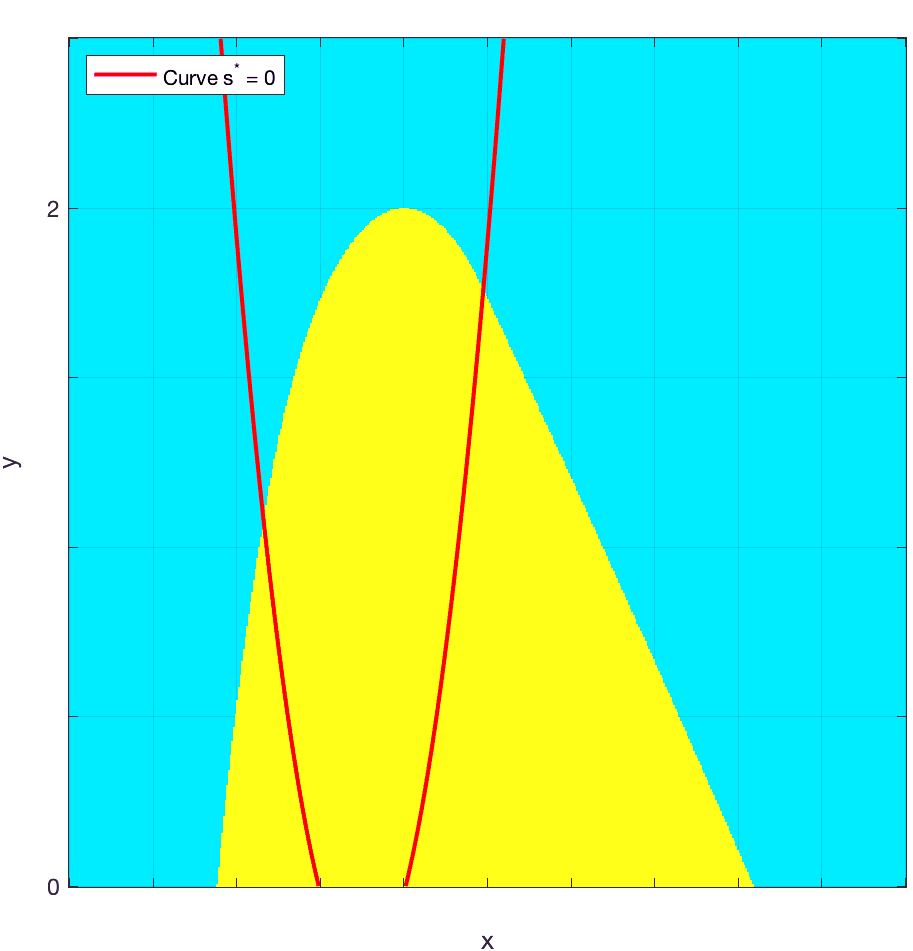}
    \caption{Optimal path topology for the fundamental solution $\varphi^*$. Panel (I): rectilinear propagation with $s^*=0$; Panel (II): field-road transition with $s^*>0$ where the optimal trajectory first reaches the road then travels along it. Dashed curve marks the critical transition boundary \(y=\frac{a}{2t}(x-ct)^2 + b(x-ct)\), which separates the rectilinear and field--road regimes.}
    \label{fig:optimal_path}
\end{figure}

Our approach synthesizes three methodologies:
\begin{enumerate}
    \item \textbf{Viscosity solution theory} for characterization and comparison principles (Section 2)
    \item \textbf{Optimal control formulation} with control triplets $(\gamma,\eta,l)$ and stopping times $\theta$ (Section 2)
    \item \textbf{Variational analysis} to identify minimizers $s^*$ and construct optimal paths (Section 3)
\end{enumerate}

The convergence $v^\varepsilon \to v$ is proved via half-relaxed limits and the dynamic programming principle (Theorem 2). The expression (1.3) follows from Freidlin's condition and explicit optimal path construction (Theorem 4).

\medskip

The Hamilton-Jacobi framework developed here extends to systems with multiple species \cite{ref4,ref5,ref6}. In particular, it enables the analysis of nonorder-preserving systems where classical comparison methods fail \cite{ref6}.

\medskip

As a natural geometric extension, Section 4 considers the problem formulated on a general \textbf{conical domain}, where two road segments intersect the field along distinct rays. This configuration effectively models scenarios in which linear transport corridors branch or converge at an angle, such as river confluences or intersecting roadways. We demonstrate that both the variational representation and the viscosity solution framework can be extended to this geometric setting, and that the enhancement in propagation induced by the roads can be characterized through analogous minimization formulas defined on the cone. Complete statements and rigorous proofs are provided in Section 4.

\section{A VARIATIONAL FORMULA AND VISCOSITY SOLUTIONS}

Consider the following variational inequality:
\begin{equation}
\begin{cases}
\min\{\mathcal{T}[v], v\} = 0, & x \in \mathbb{R},\ y > 0,\ t \in \mathbb{R}_+, \\
\mathcal{B}[v] = 0, & x \in \mathbb{R},\ y = 0,\ t \in \mathbb{R}_+, \\
v(x, y, 0) = g_0(x, y), & (x, y) \in \overline{\mathbb{R}_+^2},
\end{cases}
\tag{2.1}
\end{equation}

where
\[
\mathcal{T}[v] := v_t + |\nabla v|^2 +c v_x + 1, \quad \mathcal{B}[v] := a v_x^2 - v_y + b v_x,
\]

and the initial datum \( g_0 \) is positive, bounded, and Lipschitz continuous on \(\overline{\Rplus^2}\).

We study this problem using the theory of viscosity solutions. Denote:
\[
\mathbb{R}_+^2 := \mathbb{R} \times \mathbb{R}_+ .
\]

\begin{definition}
An upper semi-continuous function \(\underline{v}\) is a \textbf{viscosity sub-solution} of (2.1) on \(\overline{\Rplus^2} \times \mathbb{R}_+\) if for any 
\[
\phi \in C^1(\overline{\Rplus^2} \times \mathbb{R}_+) ,
\]

assuming \( \underline{v} - \phi \) attains a strict maximum at \((x, y, t) \in \overline{\Rplus^2} \times \mathbb{R}_+\) with \( \underline{v}(x, y, t) > 0 \), then:
\begin{itemize}
    \item If \( y \neq 0 \): \(\mathcal{T}[\phi](x, y, t) \leq 0\),
    \item If \( y = 0 \): \(\min\{\mathcal{T}[\phi](x, 0_+, t), B[\phi](x, 0, t)\} \leq 0\).
\end{itemize}
\end{definition}

\begin{definition}
A lower semi-continuous function \(\bar{v}\) is a \textbf{viscosity super-solution} of (2.1) if \(\bar{v} \geq 0\), and for any 
\[
\phi \in C^1(\overline{\Rplus^2} \times \mathbb{R}_+) ,
\]

assuming \(\bar{v} - \phi\) attains a strict minimum at \((x, y, t) \in \overline{\Rplus^2} \times \mathbb{R}_+\), then:
\begin{itemize}
    \item If \( y \neq 0 \): \(\mathcal{T}[\phi](x, y, t) \geq 0\),
    \item If \( y = 0 \): \(\max\{\mathcal{T}[\phi](x, 0_+, t), B[\phi](x, 0, t)\} \geq 0\).
\end{itemize}
\end{definition}

A function \( v \) is a \textbf{viscosity solution} if it is both a viscosity sub-solution and super-solution.

We state a comparison result on the whole plane and omit the proof details. The proof can be adapted from \cite{ref7} (for bounded domains) using techniques in \cite{ref3} or \cite{ref8}.

\begin{proposition}[Comparison principle for viscosity solutions]
Let \(\underline{v}\) and \(\bar{v}\) be a viscosity sub-solution and a super-solution of (2.1) respectively. Assume that \(\underline{v}\) is bounded above, \(\bar{v}\) is bounded below, and \(\underline{v}(x,y,0) \leq \bar{v}(x,y,0)\) on \(\overline{\Rplus^2}\), then \(\underline{v} \leq \bar{v}\) on \(\overline{\Rplus^2} \times \mathbb{R}_+\).
\end{proposition}

Following the approach in \cite{ref7}, we give the variational formula for the viscosity solution of (2.1). At first, we introduce an optimal control problem. For any \(p = (p_1, p_2) \in \mathbb{R}^2\), define
\[
H(p) = |p|^2 +c p_1 + 1, \quad B(p_1) = a p_1^2 + b p_1.
\]

Then, we give the Legendre transformation of \(H\) and \(B\), respectively,
\[
L(q) = \sup_{p \in \mathbb{R}^2} \left\{ q \cdot p - H(p) \right\} = \frac{|q|^2}{4} - \frac{c}{2}q_1 + \frac{c^2}{4} - 1,
\]

\[
G(q_1) = \sup_{p_1 \in \mathbb{R}} \left\{ q_1 p_1 - B(p_1) \right\} = \frac{(q_1 - b)^2}{4a}.
\]

Now, we introduce controls and stopping times. Let \(\gamma = (\gamma_1, \gamma_2) \in AC([0, t], \overline{\Rplus^2})\), \(\eta = (\eta_1, \eta_2) \in L^1([0, t], \overline{\Rplus^2})\), and \(l \in L^1([0, t], \mathbb{R})\). For \(\mathbf{x} \in \overline{\Rplus^2}\) and \(t > 0\), set following conditions:
\[
\begin{cases}
\gamma(0) = \mathbf{x}, \\
l(\tau) \geq 0 \text{ for a.e. } \tau \in [0, t], \\
l(\tau) = 0 \text{ if } \gamma_2(\tau) \neq 0 \text{ for a.e. } \tau \in [0, t],
\end{cases}
\tag{H1}
\]

and
\[
\begin{cases}
\text{the function} \quad \tau \mapsto F(\gamma, \eta, l)(\tau) := l(\tau) G\left( \frac{\eta_1 - \dot{\gamma}_1}{l} (\tau) \right) \quad \text{is integrable on } [0, t], \\
\dot{\gamma}(\tau) = \eta(\tau) \text{ if } l(\tau) = 0 \text{ for a.e. } \tau \in [0, t], \\
\eta_2(\tau) = \dot{\gamma}_2(\tau) - l(\tau) \text{ if } l(\tau) \neq 0 \text{ for a.e. } \tau \in [0, t].
\end{cases}
\tag{H2}
\]

Here, the expression \(l(\tau) G\left( \frac{\eta_1 - \dot{\gamma}_1}{l} (\tau) \right)\) in (H2) is defined only on \(\{\tau \in [0, t] : l(\tau) > 0\}\), but we understand that
\[
l(\tau) G\left( \frac{\eta_1 - \dot{\gamma}_1}{l} (\tau) \right) = 
\begin{cases}
\frac{\left(\eta_1 - \dot{\gamma}_1 - b l\right)^2 (\tau)}{4a l(\tau)} & \text{if } l(\tau) > 0, \\
\\
0 & \text{if } l(\tau) = 0.
\end{cases}
\]

Denote
\[
\mathbb{Y}^{\mathbf{x},t} := \left\{(\gamma, \eta, l) : (\gamma, \eta, l) \text{ satisfies } (H1) \text{ and } (H2)\right\}.
\]

An element \((\gamma, \eta, l)\) of \(\mathbb{Y}^{\mathbf{x},t}\) is a \textit{control triplet}.

A mapping \(\theta : L^2([0, \infty); \overline{\Rplus^2}) \mapsto [0, \infty)\) is a \textit{stopping time} if for every \(\gamma, \hat{\gamma} \in L^2([0, \infty); \overline{\Rplus^2}\) and some \(s \in [0, \infty)\), that \(\gamma(\tau) = \hat{\gamma}(\tau)\) for a.e. \(\tau \in [0, s]\) and \(\theta[\gamma] \leq s\) implies
\[
\theta[\gamma] = \theta[\hat{\gamma}].
\]

Let \(\Theta\) denote the set of all stopping times. For any \((x,t) \in \overline{\Rplus^2} \times \Rplus\), define the upper value function \(I : \overline{\Rplus^2} \times \Rplus \to \R \cup \{+\infty\}\) as
\[
I(x,t) = \sup_{\theta \in \Theta} \inf_{(\gamma, \eta, l) \in \mathbb{Y}^{\mathbf{x},t}} \left\{ \int_0^{t \wedge \theta [\gamma]} L(-\eta(\tau)) + F(\gamma, \eta, l)(\tau) d\tau + \Chi_{\{\theta [\gamma]=t\}} g_0(\gamma(t)) \right\}.
\]

One can immediately derive that \(I\) is non-negative and bounded. Moreover, it satisfies the dynamic programming principle: for any \(\sigma \in [0, t]\),
\[
I(x,t) = \sup_{\theta \in \Theta} \inf_{(\gamma, \eta, l) \in \mathbb{Y}^{\mathbf{x},t}} \left\{ \int_0^{\sigma \wedge \theta [\gamma]} L(-\eta(\tau)) + F(\gamma, \eta, l)(\tau) d\tau + \Chi_{\{\theta [\gamma] \geq \sigma\}} I(\gamma(\sigma), t - \sigma) \right\}.
\]

\begin{theorem}
The function \(I\) is continuous on \(\overline{\Rplus^2} \times \Rplus\) and a viscosity solution of (2.1). Moreover, \(\lim_{t \to 0^+} I(\mathbf{x},t) = g_0(\mathbf{x})\).
\end{theorem}

\begin{proof}
\textbf{Step 1.} For any \((\mathbf{x},t) \in \overline{\Rplus^2} \times \Rplus\), we first prove \(I\) is a viscosity sub-solution. Let \(I^*\) denote the upper semi-continuous envelope of \(I\):
\[
I^*(\mathbf{x},t) = \limsup_{(\mathbf{x'},t') \to (\mathbf{x},t)} I(\mathbf{x'},t').
\]
Assume \(I^* - \phi\) attains a strict maximum at \((\bar{x}, \bar{y}, \bar{t}) \in \overline{\Rplus^2} \times \Rplus\) with \(I^*(\bar{x}, \bar{y}, \bar{t}) > 0\). We must show:
\[
\T[\phi](\bar{x}, \bar{y}, \bar{t}) = \phi_t(\bar{x}, \bar{y}, \bar{t}) + |\nabla \phi(\bar{x}, \bar{y}, \bar{t})|^2 + c \phi_x(\bar{x}, \bar{y}, \bar{t}) + 1 \leq 0 \quad \text{if } \bar{y} \neq 0,
\]
and for \(\bar{y} = 0\):
\begin{align*}
&\min \left\{ \T[\phi](\bar{x}, 0_+, \bar{t}), \B[\phi](\bar{x}, 0, \bar{t}) \right\} = \\ 
&\min \{ \phi_t(\bar{x}, 0, \bar{t}) + |\nabla \phi(\bar{x}, 0_+, \bar{t})|^2 + c \phi_x(\bar{x}, 0_+, \bar{t}) + 1,\, a \phi_x^2(\bar{x}, 0, \bar{t}) - \phi_y(\bar{x}, 0_+, \bar{t}) + b \phi_x(\bar{x}, 0, \bar{t}) \left. \right\} \leq 0.
\end{align*}

Here, we only consider the case where \( \bar{y}=0 \); the other case can be treated similarly. 

Assume contrarily that \(\exists \epsilon > 0\) such that:
\[
\phi_t(\bar{x}, 0, \bar{t}) + |\nabla \phi(\bar{x}, 0_+, \bar{t})|^2 + c \phi_x(\bar{x}, 0_+, \bar{t}) + 1 > \epsilon \quad \text{and} \quad a \phi_x^2(\bar{x}, 0, \bar{t}) - \phi_y(\bar{x}, 0_+, \bar{t}) + b \phi_x(\bar{x}, 0, \bar{t}) > \epsilon.
\]

Without loss of generality, we may assume \( a \phi_x^2(\bar{x}, 0, \bar{t}) - \phi_y(\bar{x}, 0_+, \bar{t}) + b \phi_x(\bar{x}, 0, \bar{t}) > \epsilon / 2 \). 

Define \( \xi(x) = (2a \phi_x (x, 0, \bar{t}) + b, -1) \) and \quad \( g(x) = a \phi_x^2 (x, 0, \bar{t}) \). 

For any \( \epsilon_1 \in (0, \epsilon) \), let  
\[
U_{2\epsilon_1} := B_{2\epsilon_1} (\bar{x}, 0) \times [\bar{t} - 2\epsilon_1, \bar{t} + 2\epsilon_1].
\]

For sufficiently small \( \epsilon_1 \), assume \( \bar{t} - 2\epsilon_1 > 0 \). For all \( (\mathbf{x}, t) \in U_{2\epsilon_1} \cap (\overline{\R_+^2} \times \R_+) \):
\[
\begin{cases}
\phi_t(\mathbf{x}, t) + |\nabla \phi(\mathbf{x}, t)|^2 + c \phi_x(\mathbf{x}, t) + 1 \geq \epsilon_1, \\
\\
\xi(x) \cdot \nabla \phi (\mathbf{x}, t) - g(x) \geq 0,
\end{cases}
\tag{2.2}
\]

Since \( \phi \in C^1 (\overline{\R_+^2} \times \R_+) \), assume \( (I^* - \phi)(\bar{x}, 0, \bar{t}) = 0 \), \text{and} let \( m = -\max_{\partial U_{2\epsilon_1} \cap (\overline{\R_+^2} \times \R_+)} (I^* - \phi) \), so \( m > 0 \) and \( I \leq \phi - m \) on \( \partial U_{2\epsilon_1} \cap (\overline{\R_+^2} \times \R_+) \). Choose \( (\mathbf{x}^*, t^*) \in U_{\epsilon_1} \) such that  
\begin{align*}
(I - \phi)(x^*, t^*) > -m, \tag{2.3} \\
\quad I(x^*, t^*) > 0. \tag{2.4}
\end{align*}

Introduce auxiliary controls: assume \( (\gamma, \eta, l) \) satisfies (H1) and  
\[
\gamma(s) \in \overline{\R_+^2}, \quad \dot{\gamma}(s) = \eta(s) - l(s) \xi(\gamma_1(s)) \quad \text{for all } s \in [0, t^*]. 
\tag{\(\widetilde{H2}\)}
\]

Define  
\[
\widetilde{\mathbb{Y}}^{\mathbf{x}^*, t^*} := \left\{ (\gamma, \eta, l) : \gamma(0) = \mathbf{x}^*, \, (\gamma, \eta, l) \text{ satisfies (H1) and } (\widetilde{H2}) \right\}.
\]

One verifies \( \widetilde{\mathbb{Y}}^{\mathbf{x}^*, t^*} \subset \mathbb{Y}^{\mathbf{x}^*, t^*} \). By \cite[Proposition 4.3 and Lemma 5.5]{ref9}, \(\exists\) \( (\gamma, \eta, l) \in \widetilde{\mathbb{Y}}^{\mathbf{x}^*, t^*} \) such that for a.e. \( s \in [0, t^*] \):  
\[
|\nabla \phi (\gamma(s), t^* - s)|^2 + c \phi_x(\gamma(s), t^* - s) + \frac{|\eta(s)|^2}{4} + \frac{c}{2} \eta_1(s) + \frac{c^2}{4} \leq \epsilon_1 - \eta(s) \cdot \nabla \phi (\gamma(s), t^* - s), 
\tag{2.5}
\]

\[
\max\{|\dot{\gamma}(\tau)|, l(\tau)\} \leq C |\eta(\tau)| \quad \text{for a.e. } \tau \in [0, t^*].
\tag{2.6}
\]

Let \( \sigma = \min \{s \geq 0 : (\gamma(s), t^* - s) \in \partial U_{2\epsilon_1}\} \). Note \( \sigma \leq 3\epsilon_1 \) and \( (\gamma(s), t^* - s) \in U_{2\epsilon_1} \) for all \( 0 \leq s \leq \sigma \). 

Using (2.3) and the dynamic programming principle:  
\begin{align*}
\phi(\mathbf{x}^*, t^*) &< I(\mathbf{x}^*, t^*) + m \nonumber \\
&\leq \sup_{\theta \in \Theta} \left\{ \int_0^{\sigma \wedge \theta[\gamma]} \left( \frac{|\eta(s)|^2}{4} + \frac{c}{2} \eta_1(s) + \frac{c^2}{4} - 1 + \frac{(\eta_1 - \dot{\gamma}_1 - b l)^2(s)}{4a l(s)} \right) ds + \chi_{\theta[\gamma] \geq \sigma} I(\gamma(\sigma), t^* - \sigma) \right\} + m. \\
\tag{2.7}
\end{align*}

The supremum above requires consideration only when \( \theta[\gamma] > \sigma \) for sufficiently small \( \epsilon_1 \). If \( \theta[\gamma] < \sigma \) for sufficiently small \( \sigma \), (2.6) and \( I(\mathbf{x}^*, t^*) > 0 \) lead to a contradiction: 
\[
I(\mathbf{x}^*, t^*) \leq \sup_{\theta \in \Theta} \left\{ \int_0^{\sigma \wedge \theta |\gamma|} \left( \frac{|\eta(s)|^2}{4} + \frac{c}{2} \eta_1(s) + \frac{c^2}{4} - 1 + \frac{(\eta_1 - \dot{\gamma}_1 - b l)^2(s)}{4a l(s)} \right) ds \right\} \leq \sigma C(\eta) < \frac{I(\mathbf{x}^*, t^*)}{2}.
\]

And for \(\theta[\gamma] \geq \sigma\), using (2.7) and \((\widetilde{H2})\):
\begin{align*}
\phi(\mathbf{x}^*, t^*) <& \int_0^\sigma \left( \frac{|\eta(s)|^2}{4} + \frac{c}{2} \eta_1(s) + \frac{c^2}{4} - 1 + a l(s) \phi_x^2 (\gamma_1(s), 0, \bar{t}) \right) ds \\ 
\\
&+ I(\gamma(\sigma), t^* - \sigma) + m \\
\\
&\leq \int_0^\sigma \left( \frac{|\eta(s)|^2}{4} + \frac{c}{2} \eta_1(s) + \frac{c^2}{4} - 1 + l(s) g(\gamma_1(s)) \right) ds + \phi(\gamma(\sigma), t^* - \sigma).
\end{align*}

Thus,
\[
0 < \int_0^\sigma \left\{ \frac{|\eta(s)|^2}{4} + \frac{c}{2} \eta_1(s) + \frac{c^2}{4} - 1 + l(s) g(\gamma_1(s)) + \frac{d}{ds} \phi(\gamma(s), t^* - s) \right\} ds 
\]

\[
= \int_0^\sigma \left[ \frac{|\eta(s)|^2}{4} + \frac{c}{2} \eta_1(s) + \frac{c^2}{4} - 1 + l(s) g(\gamma_1(s)) + \nabla \phi(\gamma(s), t^* - s) \cdot \dot{\gamma}(s) - \phi_t (\gamma(s), t^* - s) \right] ds 
\]

\begin{align*}
= \int_0^\sigma & \frac{|\eta(s)|^2}{4} + \frac{c}{2} \eta_1(s) + \frac{c^2}{4} - 1 + l(s) g(\gamma_1(s)) \\
& + \nabla \phi(\gamma(s), t^* - s) \cdot [\eta(s) - l(s) \xi(\gamma_1(s))] - \phi_t (\gamma(s), t^* - s) ds.
\end{align*}

By (2.2) and (2.5):
\begin{align*}
0 < \int_0^\sigma &l(s)\left[g(\gamma_1(s)) - \nabla \phi(\gamma(s), t^* - s) \cdot \xi(\gamma_1(s))\right] \\ 
\\ 
&+ \epsilon_1 - \left(|\nabla \phi(\gamma(s), t^* - s)|^2 + c \phi_x(\gamma(s), t^* - s) + 1 + \phi_t (\gamma(s), t^* - s)\right) \, ds \leq 0.
\end{align*}

This contradiction confirms that \(I\) is a viscosity sub-solution.

\medskip

\textbf{Step 2.} We now prove \(I\) is a viscosity super-solution. 

Let \(I_\ast\) denote the lower semi-continuous envelope:
\[
I_\ast(\mathbf{x}, t) = \liminf_{(\mathbf{x}', t') \to (\mathbf{x}, t)} I(\mathbf{x}', t').
\]

For \(\phi \in C^1 (\overline{\R_+^2} \times \R_+) \), assume \(I_\ast - \phi\) attains a strict minimum at \((\bar{x}, \bar{y}, \bar{t}) \in \overline{\Rplus^2} \times \R_+\). Consider \(\bar{y} = 0\). We must show:
\[
\max \left\{ \phi_t(\bar{x}, 0, \bar{t}) + |\nabla \phi(\bar{x}, 0_+, \bar{t})|^2 + c \phi_x(\bar{x}, 0_+, \bar{t}) + 1,\, a \phi_x^2(\bar{x}, 0, \bar{t}) - \phi_y(\bar{x}, 0_+, \bar{t}) + b \phi_x(\bar{x}, 0, \bar{t}) \right\} \geq 0.
\]

Suppose contrarily that:
\[
\begin{cases}
\phi_t(\bar{x}, 0, \bar{t}) + |\nabla \phi(\bar{x}, 0_+, \bar{t})|^2 + c \phi_x(\bar{x}, 0_+, \bar{t}) + 1 < 0, \\
\\
a \phi_x^2(\bar{x}, 0, \bar{t}) - \phi_y(\bar{x}, 0_+, \bar{t}) + b \phi_x(\bar{x}, 0, \bar{t}) < 0.
\end{cases}
\]

Hence, there exists \(\epsilon > 0\) such that for all \((\mathbf{x}, t) \in U_{2\epsilon} \cap (\overline{\R_+^2} \times \R_+)\):
\[
\phi_t(\mathbf{x}, t) + |\nabla \phi(\mathbf{x}, t)|^2 + c \phi_x(\mathbf{x}, t) + 1 < 0 \quad \text{and} \quad a \phi_x^2(\mathbf{x}, \bar{t}) - \phi_y(\mathbf{x}, \bar{t}) + b \phi_x(\mathbf{x}, \bar{t}) < 0.
\tag{2.8}
\]

Assume \(\bar{t} - 2\epsilon > 0\) and \((I_\ast - \phi)(\bar{x}, 0, \bar{t}) = 0\).

Define:
\[
m := \min_{\partial U_{2\epsilon}} (I_\ast - \phi) > 0.
\]

Choose \((\mathbf{x}^\ast, t^\ast) \in U_\epsilon\) such that \((I - \phi)(\mathbf{x}^\ast, t^\ast) < m\). By definition of \(I\), select \(\theta \equiv t^\ast\) and \((\gamma, \eta, l) \in \mathbb{Y}^*\) satisfying:  
\[
I(\mathbf{x}^\ast, t^\ast) + m \geq \int_0^{t^\ast} \left( \frac{|\eta(s)|^2}{4} + \frac{c}{2} \eta_1(s) + \frac{c^2}{4} - 1 + \frac{(\eta_1 - \dot{\gamma}_1 - b l)^2(s)}{4a l(s)} \right) ds + g_0(\gamma(t^\ast)).
\]

Let \(\sigma = \min\{s \geq 0: (\gamma(s), t^\ast - s) \in \partial U_{2\epsilon}\}\). 

Then:  
\[
\phi(\mathbf{x}^\ast, t^\ast) + m > \int_0^{\sigma} \left( \frac{|\eta(s)|^2}{4} + \frac{c}{2} \eta_1(s) + \frac{c^2}{4} - 1 + \frac{(\eta_1 - \dot{\gamma}_1 -b l)^2(s)}{4a l(s)} \right) ds + I(\gamma(\sigma), t^\ast - \sigma)
\]

\[
\geq \int_0^{\sigma} \left( \frac{|\eta(s)|^2}{4} + \frac{c}{2} \eta_1(s) + \frac{c^2}{4} - 1 + \frac{(\eta_1 - \dot{\gamma}_1 - b l)^2(s)}{4a l(s)} \right) ds + \phi(\gamma(\sigma), t^\ast - \sigma) + m
\]

\[
\geq \int_0^{\sigma} \left( \frac{|\eta(s)|^2}{4} + \frac{c}{2} \eta_1(s) + \frac{c^2}{4} - 1 + \frac{(\eta_1 - \dot{\gamma}_1 - b l)^2(s)}{4a l(s)} + \frac{d}{ds} \phi(\gamma(s), t - s) \right) ds + \phi(\mathbf{x}^*, t^\ast) + m
\]

\begin{align*}
&\geq \int_0^{\sigma} \left( \frac{|\eta(s)|^2}{4} + \frac{c}{2} \eta_1(s) + \frac{c^2}{4} - 1 + \frac{(\eta_1 - \dot{\gamma}_1 - b l)^2(s)}{4a l(s)} \right. \\
\\
&\quad \left. + \nabla \phi(\gamma(s), t^* - s) \cdot \dot{\gamma}(s) - \phi_t (\gamma(s), t^* - s) \right) ds + \phi(\mathbf{x}^*, t^\ast) + m.
\end{align*}

By (H2) and Young's inequality, for almost every \(s \in \{s: \gamma_2(s) \neq 0\}\):  
\[
\frac{|\eta(s)|^2}{4} + \frac{c}{2} \eta_1(s) + \frac{c^2}{4} + \frac{(\eta_1 - \dot{\gamma}_1 - b l)^2(s)}{4a l(s)} \geq -\dot{\gamma}(s) \cdot \nabla \phi(\gamma(s), t^\ast - s) - |\nabla \phi(\gamma(s), t^\ast - s)|^2 - c \phi_x(\gamma(s), t^\ast - s) \\
\tag{2.9}
\]

For \( a.e. s \in \{s: \gamma_2(s) = 0\}\) (where \(\dot{\gamma}_2(s) = 0\), \(\eta_2(s) = - l(s)\)):  
\[
\frac{|\eta(s)|^2}{4} + \frac{c}{2} \eta_1(s) + \frac{c^2}{4} + \frac{(\eta_1 - \dot{\gamma}_1 - b l)^2(s)}{4a l(s)} = \frac{\eta^2_1(s) + l^2(s)}{4} + \frac{c}{2} \eta_1(s) + \frac{c^2}{4} + \frac{(\eta_1 - \dot{\gamma}_1 - b l)^2(s)}{4a l(s)}
\]

\[
\geq -\dot{\gamma}_1(s) p_1(s) - |\mathbf{p}(s)|^2 - c p_1(s) - l(s)(ap_1^2(s) + b p_1(s) - p_2(s)). 
\tag{2.10}
\]

where \( \mathbf{p} = (p_1, p_2) := (\phi_x(\gamma(s), t^* - s), \phi_y(\gamma_1(s), 0+, t^* - s)) \). 

Combining (2.8), (2.9), and (2.10):  
\begin{align*}
0 &> \int_0^{\sigma} \left( \frac{|\eta(s)|^2}{4} + \frac{c}{2} \eta_1(s) + \frac{c^2}{4} - 1 + \frac{(\eta_1 - \dot{\gamma}_1 - b l)^2(s)}{4a l(s)} \right. \\
\\
&\quad \left. + \nabla \phi(\gamma(s), t^* - s) \cdot \dot{\gamma}(s) - \phi_t (\gamma(s), t^* - s) \right) ds
\end{align*}
\[
\geq \int_{0}^{\sigma} -\left(\phi_t(\gamma(s), t^\ast - s) + |\nabla \phi(\gamma(s), t^\ast - s)|^2 + c \phi_x(\gamma(s), t^* - s) + 1\right) - l(s)\left(ap_1^2(s) + b p_1(s) - p_2(s)\right) ds \geq 0.
\]

This contradiction confirms \(I\) is a viscosity super-solution.

\medskip

\textbf{Step 3.} To prove \(I\) is continuous on \(\overline{\Rplus^2} \times [0, \infty)\), by Proposition 2.2, we need to show: 
\[
I^\ast(\mathbf{x}, 0) \leq I_\ast(\mathbf{x}, 0) \quad \forall \mathbf{x} \in \overline{\Rplus^2}.
\]

Following \cite[Theorem 4.2]{ref7}, select \(g_0^\epsilon \in C^1(\overline{\Rplus^2}) \) satisfying:  
\[
a(g_{0,x}^\epsilon)^2(x, 0) + b\, g_{0,x}^\epsilon(x, 0) - g_{0,y}^\epsilon(x, 0+) \leq 0, \quad \text{and} \quad |g_0(\mathbf{x}) - g_0^\epsilon(\mathbf{x})| \leq \epsilon \quad \forall x \in \overline{\Rplus^2}.
\]

Choose a sufficiently large constant \(C_\epsilon > 0\) such that \(\psi(\mathbf{x}, t) := g_0^\epsilon(\mathbf{x}) - C_\epsilon t\) satisfies:  
\[
\psi_t(\mathbf{x}, t) + |\nabla \psi(\mathbf{x}, t)|^2 + c \psi_x(\mathbf{x}, t) + 1 \leq 0 \quad \text{on } \overline{\Rplus^2}.
\]

For any \((\gamma, \eta, l) \in Y^{x, t}\):  
\[
\psi(\gamma(t), 0) - \psi(\mathbf{x}, t) = \int_{0}^{t} \nabla \psi(\gamma(\tau), t - \tau) \cdot \dot{\gamma}(\tau) - \psi_t(\gamma(\tau), t - \tau) d\tau.
\]

For \(a.e. \,\tau \in \{\gamma_2(\tau) \neq 0\}\):  
\begin{align*}
\nabla \psi(\gamma(\tau), t - \tau) \cdot \dot{\gamma}(\tau) - \psi_t(\gamma(\tau), t - \tau) &\geq \nabla \psi(\gamma(\tau), t - \tau) \cdot \dot{\gamma}(\tau) + |\nabla \psi(\gamma(\tau), t - \tau)|^2 + c \psi_x(\gamma(\tau), t - \tau) + 1 \\
\\
&\geq - L(-\eta(\tau)) - F(\gamma, \eta, l)(\tau). 
\end{align*}

For \(a.e. \,\tau \in \{\gamma_2(\tau) = 0\}\) \text{\, and \,} using \(a \psi_x^2(x, 0, t) + b \,\psi_x(x, 0, t) - \psi_y(x, 0_+, t) \leq 0\;\):
\begin{align*}
& \quad \psi_x(\gamma(\tau), t - \tau) \dot{\gamma}_1(\tau) - \psi_t(\gamma(\tau), t - \tau) \\
\\
&\geq \psi_x(\gamma(\tau), t - \tau) \dot{\gamma}_1(\tau) + |\nabla \psi(\gamma(\tau), t - \tau)|^2 + c \psi_x(\gamma(\tau), t - \tau) + 1 \\
\\
&\geq \psi_x(\gamma(\tau), t - \tau) \dot{\gamma}_1(\tau) + |\nabla \psi(\gamma(\tau), t - \tau)|^2 + c \psi_x(\gamma(\tau), t - \tau) + 1 \\
\\
& \quad + l(\tau)(a \psi_x^2(\gamma(\tau), t - \tau) + b \psi_x(\gamma(\tau), t - \tau) - \psi_y(\gamma_1(\tau), 0_+, t-\tau)) \\
\\
&\geq - L(-\eta(\tau)) - F(\gamma, \eta, l)(\tau).
\end{align*}

Thus:  
\[
\psi(\mathbf{x},t) \leq \int_0^t  L(-\eta(\tau)) + F(\gamma, \eta, l)(\tau) d\tau + g_0^\epsilon(\mathbf{x}).
\]

As \((\gamma, \eta, l)\) is arbitrary, we have: 
\[
\psi(\mathbf{x},t) \leq \inf_{(\gamma, \eta, l) \in \mathbb{Y}^{\mathbf{x},t}}  \int_0^t L(-\eta(\tau)) + F(\gamma, \eta, l)(\tau) d\tau + g_0^\epsilon(\mathbf{x}) \leq I(\mathbf{x},t),
\]

Hence \(g_0(\mathbf{x}) - 2\epsilon \leq I_\ast(\mathbf{x},0)\) on \(\overline{\Rplus^2}\). Letting \(\epsilon \to 0\), we get \(g_0(\mathbf{x}) \leq I_\ast(\mathbf{x},0)\) on \(\overline{\Rplus^2}\).

Conversely, take \((\gamma, \eta, l) = (\mathbf{x}, 0, 0) \in \mathbb{Y}^{x,t}\):  
\[
I(\mathbf{x},t) \leq \sup_{\theta \in \Theta}  \int_0^{\theta[\gamma]} L(0) + \Chi_{\{\theta[\gamma] = t\}} g_0(\mathbf{x}) \leq g_0(\mathbf{x}) \quad \forall (\mathbf{x}, t) \in \overline{\Rplus^2} \times \Rplus
\]

Thus \(I^*(\mathbf{x},0) \leq g_0(\mathbf{x})\). 

Combining both results yields 
\[
I^\ast(\mathbf{x}, 0) \leq I_\ast(\mathbf{x}, 0) \quad \forall x \in \overline{\Rplus^2}.
\]

\end{proof}

\section{ESTIMATES AND LIMITS OF THE PHASE FUNCTION}

\subsection{Convergence of the Phase Function}

In this section, we will show that \(v^{\epsilon}\) converges to the viscosity solution of (1.1).

Initially, certain uniform estimates pertaining to \( u^\varepsilon \) and \( v^\varepsilon \) will be obtained. Analogous to the demonstration of \cite[Lemma 1.2]{ref2}, the subsequent lemma is easily derivable through application of the comparison principle \cite[Proposition 4.2]{ref1}; ergo, we omit a detailed exposition herein.

\begin{lemma}
There subsists a constant \( C \), independent of \( \varepsilon > 0 \), such that  
\[
0 < u^\varepsilon \leq C, \quad \forall (x, y, t) \in \overline{\Rplus^2} \times \mathbb{R}_+.
\]

Furthermore,  
\[
\limsup_{\varepsilon \to 0} u^\varepsilon \leq 1
\]

locally uniformly in \( \overline{\Rplus^2} \times \mathbb{R}_+ \).
\end{lemma}

\begin{lemma}
For any compact subset \( K \subset \overline{\Rplus^2} \times \mathbb{R}_+ \), there subsists \( C(K) \), independent of \( \varepsilon \), such that  
\[
\sup_{K} |v^\varepsilon| \leq C(K).
\]
\end{lemma}

Having established the \( C_{\text{loc}} \) bounds, we define the \textbf{half-relaxed limits}:
\[
v^*(\mathbf{x},t) = \limsup_{\substack{\varepsilon \to 0+ \\ (\mathbf{x}',t') \to (\mathbf{x},t)}} v^\varepsilon(\mathbf{x}',t'), \quad v_*(\mathbf{x},t) = \liminf_{\substack{\varepsilon \to 0+ \\ (\mathbf{x}',t') \to (\mathbf{x},t)}} v^\varepsilon(\mathbf{x}',t').
\]

Observe that
\begin{itemize}
\item \( v^* \) is bounded and upper semi-continuous.
\item \( v_* \) is bounded and lower semi-continuous.
\item \( v^* \geq v_* \geq 0 \) by Lemma 1.
\end{itemize}

We shall demonstrate \( v^* = v_* = v \), where \( v \) is defined via the variational formula:
\[
v(\mathbf{x},t) = \sup_{\theta \in \Theta} \inf_{(\gamma, \eta, l) \in \mathbb{Y}^{\mathbf{x},t}} \left\{ \int_0^{t \wedge \theta[\gamma]} L(-\eta(\tau)) + F(\gamma, \eta, l)(\tau) d\tau \,\bigg|\, \gamma(t) = (0,0) \right\}.
\]

\begin{theorem}

The half-relaxed limits \( v^* \) and \( v_* \) satisfy:
\begin{enumerate}
\item \( v^* \) is a viscosity sub-solution of (1.1) on \( \overline{\Rplus^2} \times \R_+ \).
\item \( v_* \) is a viscosity super-solution of (1.1) on \( \overline{\Rplus^2} \times \R_+ \).
\item \( v^*(\mathbf{x},t) \leq v(\mathbf{x},t) \leq v_*(\mathbf{x},t) \) for all \( (\mathbf{x},t) \in \overline{\Rplus^2} \times (0,\infty) \).
\end{enumerate}

Consequently, \( v \) is the unique viscosity solution of (1.1).

\end{theorem}

\begin{proof}

\textbf{Step 1.} Following \cite[Lemma 2.2]{ref2} and \cite[Theorem 3.1]{ref10}, \( v^* \) is a sub-solution of \( \min\{\mathcal{T}[v], v\} = 0 \) in \( \overline{\Rplus^2} \times \R_+ \). We verify boundary conditions in the viscosity sense: for \( \phi \in C^1(\overline{\Rplus^2} \times \R_+) \), if \( v^* - \phi \) attains a strict maximum at \( (x,0,t) \in \overline{\Rplus^2} \times \R_+ \) with \( v^*(x,0,t) > 0 \), then:
\[
\min\left\{\mathcal{T}[\phi](x,0_+,t),\, \mathcal{B}[\phi](x,0,t)\right\} \leq 0.
\]

For \( \phi \in C^2(\overline{\R^2_+} \times \R_+) \), \(\exists\) \( (\mathbf{x}_\epsilon, t_\epsilon) \to (x,0,t) \in \overline{\Rplus^2} \times \Rplus \) as \(\epsilon \to 0_+\) such that \( v^\epsilon - \phi \) attains a local maximum. If \( \{\mathbf{x}_\epsilon\} \) has a subsequence in \( \overline{\Rplus^2} \), then \( \mathcal{T}[\phi](x,0_+,t) \leq 0 \) by \cite[Lemma 2.2]{ref2}. For \( \mathbf{x}_\epsilon \in \{y=0\} \):
\[
\begin{cases}
(v^\epsilon - \phi)_{xx} \leq 0, \\
(v^\epsilon - \phi)_x = 0, \\
(v^\epsilon - \phi)_y \geq 0 \quad \text{at } (\mathbf{x}_\epsilon,t_\epsilon),
\end{cases}
\]

implying:
\[
a(\phi_x)^2 - \phi_y + b \phi_x \leq a(v_x^\epsilon)^2 - v_y^\epsilon + b v_x^\epsilon = a\epsilon v_{xx}^\epsilon \leq a\epsilon \phi_{xx}.
\]

As \( \epsilon \to 0^+ \), \( \mathcal{B}[\phi](x,0,t) \leq 0 \). 

For \(\phi \in C^1(\overline{\Rplus^2} \times \R_+)\), construct \(\phi_n \in C^2(\overline{\Rplus^2} \times \R_+)\) such that:
\[
\phi_n \to \phi \quad \text{in } C^1(\overline{\Rplus^2}).
\]

For sufficiently large \(n\), \(v^\varepsilon - \phi_n\) attains a local maximum at \((\mathbf{x}_n, t_n) \in \overline{\Rplus^2} \times \R_+\) with:
\[
(\mathbf{x}_n, t_n) \to (x, 0, t) \quad \text{as } n \to \infty.
\]

Noting \(D\phi_n(\mathbf{x}_n, t_n) \to D\phi(x, 0, t)\), we conclude:
\[
v^* \text{ is a viscosity sub-solution of (2.1) on } \overline{\Rplus^2} \times \R_+.
\]

Analogously, \(v_*\) is shown to be a viscosity super-solution through parallel arguments.

\textbf{Step 2.} Now we prove
\[
v_*(\mathbf{x},t) \geq v(\mathbf{x},t). 
\tag{3.1}
\]

Let \( \zeta: \overline{\Rplus^2} \to \R \) satisfy:
\[
0 \leq \zeta \leq 1,\quad \zeta(0,0) = 0,\quad \zeta > 0 \text{ on } \overline{\Rplus^2} \setminus \{(0,0)\}.
\]

Define for \( k > 0 \):
\[
g_k^\epsilon(\mathbf{x}) = g^\epsilon(\mathbf{x}) + e^{-k\zeta(\mathbf{x})/\epsilon}.
\]

Let \( u_k^\epsilon \) satisfy (1.3) with initial data \( g_k^\epsilon \). By comparison principle:
\[
u_k^\epsilon \geq u^\epsilon \implies v_k^\epsilon = -\epsilon \log u_k^\epsilon \leq v^\epsilon.
\]

The function \( v_k^\epsilon \) satisfies:
\[
\begin{cases}
v_{k,t}^\epsilon - \epsilon \Delta v_k^\epsilon + |\nabla v_k^\epsilon|^2 +c v_{k,x}^\epsilon + 1 - e^{-v_k^\epsilon/\epsilon} = 0, & \quad (x,t) \in \overline{\Rplus^2} \times \R_+ \setminus \Gamma, \\
-a\epsilon v_{k,xx}^\epsilon + a\left(v_{k,x}^\epsilon\right)^2 + b\,v_{k,x}^\epsilon - v_{k,y}^\epsilon = 0, & \quad (x,t) \in \Gamma, \\
v_k^\epsilon(\mathbf{x},0) = 
\begin{cases}
-\epsilon \log g_k^\epsilon(\mathbf{x}), & \quad \mathbf{x} \in G_\epsilon, \\
k\zeta(\mathbf{x}), & \quad \mathbf{x} \in \overline{\Rplus^2} \setminus G_\epsilon.
\end{cases}
\end{cases}
\]

Analogous to Lemma 2, \(\{v_k^\varepsilon\}\) is bounded in \(C_{\text{loc}}(\overline{\Rplus^2} \times \R_+)\) uniformly in \(\varepsilon\). Hence, the half-relaxed limits \(v_k^*\) and \(v_{*}^k\) exist and are respectively sub- and super-solutions of (2.1) on \(\overline{\Rplus^2} \times \R_+\). Noting:
\[
v_k^*(\mathbf{x},0) = v_{*}^k(\mathbf{x},0) = k\zeta(\mathbf{x}),
\]

Proposition 1 implies \(v_k = v_k^* = v_{*}^k\) on \(\overline{\Rplus^2} \times [0,\infty)\), where \(v_k\) solves (2.1) with \(g_0(\mathbf{x}) = k\zeta(\mathbf{x})\). By Theorem 1:
\[
v_k(\mathbf{x},t) = \sup_{\theta \in \Theta} \inf_{(\gamma,\eta,l) \in \mathbb{Y}^{\mathbf{x},t}} \left\{ \int_0^{\theta[\gamma]} L(-\eta(\tau)) + F(\gamma,\eta,l)(\tau) d\tau + \Chi_{\{\theta[\gamma] = t\}} k\zeta(\mathbf{x}) \right\}.
\]

Since \(v_k^\varepsilon \leq v^\varepsilon\), we have \(v_k \leq v_{\ast}\) for all \(k > 0\). Letting \(k \to \infty\), (3.1) follows.

\textbf{Step 3. } To prove
\[
v^*(\mathbf{x},t) \leq v(\mathbf{x},t), 
\tag{3.2}
\]

observe that for any fixed \(\delta, R>0\), \(v^*\) is a sub-solution of:
\[
\begin{cases}
\min\left\{v_t + |\nabla v|^2 + c v_x + 1,\, v\right\} = 0, & |x| < R,\, 0 < |y| < R,\, t > \delta, \\
a v_x^2(x,0,t) - v_y(x,0_+,t) + b\,v_x(x,0,t) = 0, & |x| < R,\, y = 0,\, t > \delta.
\end{cases}
\]

Following \cite[Lemma 3.1]{ref2}, we have
\begin{align*}
v^*(\mathbf{x},t) \leq \sup_{\theta \in \Theta} \inf_{(\gamma,\eta,l) \in \mathbb{Y}^{\mathbf{x},t}} \biggl\{ &\int_0^{(t-\delta) \wedge s \wedge \theta[\gamma]} L(-\eta(\tau)) + F(\gamma,\eta,l)(\tau) \, d\tau \\
&+ \Chi_{\{s \wedge \theta[\gamma] \geq (t-\delta)\}} v^*(\gamma(t-\delta),\delta) \\
&+ \Chi_{\{(t-\delta) \wedge \theta[\gamma] \geq s\}} v^*(\gamma(s),t-s) \biggr\},
\end{align*}

where \(s := \inf\{\tau : \gamma(\tau) \in \partial B_R(0)\}\) is the so-called exit time from \(int B_R(0)\).

Direct computation yields:
\begin{align*}
v^*(\mathbf{x},t) \leq \sup_{\theta \in \Theta} \inf_{(\gamma,\eta,l) \in \mathbb{Y}^{\mathbf{x},t}} \Bigg\{ 
& \int_0^{(t-\delta) \wedge \theta[\gamma]}  L(-\eta(\tau)) + F(\gamma,\eta,l)(\tau) \, d\tau 
\\
& + \Chi_{\{\theta[\gamma] \geq (t-\delta)\}} v^*(\gamma(t-\delta),\delta) 
\,\Bigg|\, 
\begin{aligned}
&\gamma(\tau) \in B_R(0) \text{ for } 0 \leq \tau \leq t - \delta, \\
&\gamma(t-\delta) = (0,0)
\end{aligned} 
\Bigg\}.
\end{align*}

Letting \(R \to \infty\) and \(\delta \to 0\), (3.2) is established.

\end{proof}

\begin{remark}

\begin{enumerate}
\item The limit \( v \) constitutes the unique viscosity solution of (2.1) within the class of functions \( f \) satisfying:
\[
\begin{cases}
f \text{ is bounded below and locally Lipschitz continuous on } (\overline{\Rplus^2} \times \R_+) \cup \{(0,0,0)\}, \\
f(0,0,0) = 0, \quad f(\mathbf{x},t) \to \infty \text{ as } t \to 0_+ \text{ on } \overline{\Rplus^2} \setminus \{(0,0)\}.
\end{cases}
\]
As these arguments follow standard methodology, we defer to the comprehensive treatments in \cite{ref7,ref11}.

\item From the definition of half-relaxed limits and \( v^* = v = v_* \), one immediately deduces that \( v^\varepsilon \) converges to \( v \) uniformly on any compact subset \( K \subset \overline{\Rplus^2} \). (Full technical details appear in \cite[Remark 6.4]{ref10})
\end{enumerate}

\end{remark}

\subsection{Formulation for the Asymptotic Behaviour of the Phase Function}

In the previous section, we have established that  \( v^{\epsilon} \) converges to the viscosity solution \(v\) of (1.1) and presented the variational formula for \(v\). We now demonstrate the formula (1.3). Define the \textbf{payoff}: 
\[
J(\mathbf{x},t) = \inf_{(\gamma,\eta,l)\in \mathbb{Y}^{\mathbf{x},t}} \left\{ \int_0^t L(-\eta(\tau)) + F(\gamma,\eta,l)(\tau) d\tau \,\bigg|\, \gamma(t) = (0,0) \right\}.
\tag{3.3}
\]

\( J \) is the viscosity solution of:
\[
\begin{cases}
v_t + |\nabla v|^2 + cv_x + 1 = 0, & (x,y,t) \in \overline{\Rplus^2} \times \mathbb{R}_+ \setminus \Gamma, \\
a v_x^2 - v_y + b\,v_x = 0, & (x,y,t) \in \Gamma, \\
v(x,y,0) = 0, & (x,y) = (0,0), \\
v(x,y,t) \to \infty & \text{as} \quad t \to 0^+, \quad (x,y) \in \overline{\Rplus^2} \setminus \{(0,0)\}.
\end{cases}
\]

\begin{theorem}
The function \( J \) can be expressed as:
\[
J(\mathbf{x},t) = \varphi^*(\mathbf{x},t) - t, 
\tag{3.4}
\]

where 
\[
\varphi^*(\mathbf{x},t) = \min_{s\geq 0} \left\{ \frac{\left(-x + b\,s + c\,t\right)^2}{4(t+as)} + \frac{(y+s)^2}{4t} \right\}.
\]
\end{theorem}

\bigskip

\begin{proposition}
For any \((\mathbf{x},t) \in \overline{\Rplus^2} \times \R_+\), there exists a unique \( s^* = s^*(\mathbf{x}, t) \geq 0 \) satisfying:
\[
\varphi^*(\mathbf{x},t) = \frac{\left(-x + b\,s^\ast + c\,t\right)^2}{4(t+as^\ast)} + \frac{(y+s^\ast)^2}{4t}.
\]

Moreover, we have the following properties:
\begin{enumerate}
    \item \( s^* = 0 \) if and only if \( (\mathbf{x}, t) \in S_0 \coloneqq \{(\mathbf{x}, t) \in \overline{\Rplus^2} \times \R_+ : y \geq \dfrac{a}{2t}(x-ct)^{2} + b(x-ct)\} \). For \( (\mathbf{x}, t) \in S \coloneqq (\overline{\Rplus^2} \times \R_+) \setminus S_0 \), \( \quad s^* \) satisfies the equation:
    \[
    2(y + s^\ast)(t + as^\ast)^2 = ax^2t - 2(ac-b)x t^2 + c(ac-2b) t^3.
    \tag{$\star$}
    \]
    
    \item The function \( \varphi^*(\mathbf{x}, t) \) exhibits \( \mathcal{C}^\infty \)-smoothness in \( S_0 \cap (\overline{\Rplus^2} \times \R_+) \) and \( S \cap (\overline{\Rplus^2} \times \R_+) \), whilst maintaining \( \mathcal{C}^1 \)-smoothness throughout \( \overline{\Rplus^2} \times \R_+ \).
    
    \item Both \( \varphi^*(\mathbf{x}, t) \) and \( s^*(\mathbf{x}, t) \) demonstrate homogeneity of degree one in the variable \(t\), i.e. :
    \[
    \varphi^*(t\mathbf{x}, t) = t \varphi^*(\mathbf{x}, 1), \quad s^*(t\mathbf{x}, t) = t s^*(\mathbf{x}, 1).
    \]
\end{enumerate}
\end{proposition}

\begin{proof}

\medskip

Define the objective function:
\[
f(s;\mathbf{x},t) \coloneqq \frac{\left(-x + b s + c t\right)^2}{4(t+as)} + \frac{(y+s)^2}{4t},
\]
so that $\varphi^*(\mathbf{x},t) = \min_{s \geq 0} f(s;\mathbf{x},t)$.

\medskip

\textbf{Existence.} Fix \((x,y,t)\) with \(t>0\). The function \(s\mapsto f(s;x,y,t)\) is continuous on \([0,\infty)\) and satisfies
\(\lim_{s\to\infty} f(s;x,y,t)=\infty\), since \((y+s)^2/(4t)\sim s^2/(4t)\to\infty\). Hence a minimum exists on \([0,\infty)\).

\medskip

\noindent\textbf{Uniqueness via strict convexity.} Compute the second derivative with respect to \(s\). A direct calculation yields the closed form
\[
\frac{\partial^2 f}{\partial s^2}
= \frac{(a s + t)^3 + t\bigl(a(x-ct)+bt\bigr)^2}{2t\,(a s + t)^3}.
\]

Since \(t>0\) and \(a s + t>0\) for all \(s\ge0\), the numerator is strictly positive; hence \(\partial_{ss} f>0\) for all \(s\ge0\). Thus \(f\) is strictly convex on \([0,\infty)\), and the minimizer is unique.

\medskip

\noindent\textbf{Characterization of the minimizer and $S_0$.} Let \(f_s=\partial f/\partial s\). Differentiation gives
\[
f_s = -\frac{a\bigl(-x+b s + c t\bigr)^2}{4(t+as)^2} + \frac{b\bigl(-x+b s + c t\bigr)}{2(t+as)} + \frac{y+s}{2t}.
\]

If the unique minimizer satisfies \(s^\ast=0\), the necessary and sufficient condition is \(f_s(0)\ge0\). Evaluating at \(s=0\) we obtain
\[
f_s(0) = -\frac{a(c t - x)^2}{4t^2} + \frac{b(c t - x)}{2t} + \frac{y}{2t}.
\]

Multiplying by \(4t^2>0\) and rearranging gives the equivalent inequality
\[
- a(c t - x)^2 + 2t\bigl(b(c t - x) + y\bigr)\ge0,
\]

which, after replacing \(c t - x = -(x-ct)\), is exactly
\[
y \ge \frac{a}{2t}(x-ct)^2 + b(x-ct).
\]

Thus the threshold set \(S_0\) is as stated.

If \(f_s(0)<0\) then the (unique) minimizer lies in the interior \(s^\ast>0\) and satisfies \(f_s(s^\ast)=0\). Multiplying \(f_s(s^\ast)=0\) by \(\; 4(t+as^\ast)^2\cdot 2t \;\) to clear denominators yields the identity
\[
2(y + s^\ast)(t+ a s^\ast)^2
= a\bigl(-x + b s^\ast + c t\bigr)^2 t - 2b\bigl(-x + b s^\ast + c t\bigr)(t+ a s^\ast) t,
\]

which is the statement \((\star)\) in the proposition and algebraically equivalent to the cubic equation
\begin{align*}
2a^2 (s^\ast)^3 &+ \bigl(2a^2 y + a b^2 t + 4a t\bigr)(s^\ast)^2 \\
&+ \bigl(4a t y + 2b^2 t^2 + 2 t^2\bigr) s^\ast \\
&+ \bigl(-a c^2 t^3 + 2a c t^2 x - a t x^2 + 2 b c t^3 - 2 b t^2 x + 2 t^2 y\bigr) = 0.
\end{align*}

Therefore \(s^\ast\) is the unique real root \(\ge0\) of that cubic.

\medskip

\noindent\textbf{Regularity.} On \(S_0\) we have \(s^\ast\equiv0\) and
\[
\varphi^*(x,y,t)=f(0;x,y,t)=\frac{(-x+ct)^2+y^2}{4t},
\]

which is \(C^\infty\) for \(t>0\). On the complementary region \(S\) the minimizer satisfies the smooth implicit equation \(f_s(s;x,y,t)=0\) and, since \(f_{ss}(s^\ast)>0\), the Implicit Function Theorem gives that \(s^\ast(x,y,t)\) is \(C^\infty\) there; hence \(\varphi^*=f(s^\ast(\cdot),\cdot)\) is \(C^\infty\) on \(S\). On the boundary \(\partial S_0\) the one-sided derivatives computed from the two sides agree (verify by direct substitution \(s^\ast=0\) in the envelope expressions), so \(\varphi^*\) is globally \(C^1\).

\medskip

\noindent\textbf{Homogeneity.} Substitute \(s=k\,t\) and compute
\[
\begin{aligned}
\varphi^*(t x,t y,t)
&=\min_{k\ge0} \left\{ \frac{(-t x + bkt + c t)^2}{4(t + a k t)} + \frac{(t y + k t)^2}{4t}\right\}\\
&=t\min_{k\ge0}\left\{\frac{(-x + b k + c)^2}{4(1 + a k)} + \frac{(y + k)^2}{4}\right\}
= t\,\varphi^*(x,y,1).
\end{aligned}
\]

Hence \(\varphi^*(t\mathbf{x},t)=t\varphi^*(\mathbf{x},1)\) and the minimizer scales as \(s^\ast(t\mathbf{x},t)=t s^\ast(\mathbf{x},1)\).

\medskip

This completes the proof.
\end{proof}

\bigskip

\textbf{Proof of Theorem 3 :}

\begin{proof}
For any triplets \((\gamma, \eta, l) \in \mathbb{Y}^{\mathbf{x},t}\) and any fixed \((\mathbf{x}, t) \in \overline{\Rplus^2} \times \Rplus\), we have
\begin{align*}
&\quad\int_{0}^{t} L(-\eta(\tau)) + F(\gamma,\eta,l)(\tau) \, d\tau \\ 
\\
= &\int_{0}^{t} \frac{\eta_1^2 + \eta_2^2}{4} + \frac{c}{2} \eta_1 + \frac{c^2}{4} - 1 + \frac{(\eta_1 - \dot{\gamma}_1 - b l)^2}{4al} \, d\tau \\
\\
= &\int_{0}^{t} \frac{(al+1)\eta_1^2 + 2(acl-\dot{\gamma}_1 - b l)\eta_1 + alc^2 + (\dot{\gamma}_1 + b l)^2}{4al} + \frac{(\dot{\gamma}_2 - l)^2}{4} \, d\tau - t \\
\\
\geq &\int_{0}^{t} \frac{(\dot{\gamma}_1 + bl + c)^2}{4(1+al)} + \frac{(\dot{\gamma}_2 - l)^2}{4} \, d\tau - t
\end{align*}

The condition for equality is
\[
\dot{\gamma}_1 = (1+al)\eta_1 + (ac-b)l.
\]

Define convex functions on \(\Rplus^2\):
\[
K_1(m,l) := \frac{\left(m+bl+c\right)^2}{4(1+al)}, \quad K_2(m,l) := \frac{(m - l)^2}{4}.
\]

By Jensen's inequality:
\begin{align*}
& \quad \int_{0}^{t} \frac{(\dot{\gamma}_1 + bl+c)^2}{4(1+al)} + \frac{(\dot{\gamma}_2 - l)^2}{4} \, d\tau \\
\\
&= \int_0^t K_1(\dot{\gamma}_1(\tau),l(\tau)) + K_2(\dot{\gamma}_2(\tau),l(\tau)) d\tau \\
\\
&\geq t \left[ K_1\left(\textless \dot{\gamma}_1 \textgreater, \textless l \textgreater\right) + K_2\left(\textless \dot{\gamma}_2 \textgreater, \textless l \textgreater\right) \right] \\
\\
&= \frac{\left( -x + b\,s + c\,t \right)^2}{4(t+as)} + \frac{(y+s)^2}{4t}
\end{align*}

where \( s = \|l\|_1 \quad and \quad \textless f \textgreater = \frac{1}{t} \int_0^t f(\tau) d\tau \). 

Thus:
\[
J(x,y,t) \geq \min_{s \geq 0} \left\{ \frac{\left(-x + b\,s + c\,t \right)^2}{4(t+as)} + \frac{(y+s)^2}{4t} \right\} - t.
\]

Now, we are going to establish that this infimum can be attained. 

By Proposition 2, there exists an unique \(s^\ast = s^\ast(x, y, t) \geq 0\) such that 
\[
\varphi^*(x, y ,t) = \frac{\left(-x + b\,s^\ast + c\,t\right)^2}{4(t+as^\ast)} + \frac{(y+s^\ast)^2}{4t}.
\]

Consider the following control triplets:
\[
\eta(\tau) = \left( -\frac{x + (ac-b)s^*}{t + as^*}, \quad -\frac{y + s^*}{t} \right), \quad 0 \leq \tau \leq t,
\]

\[
l(\tau) = 
\begin{cases}
0, & 0 \leq \tau < t_0, \\
\\
\frac{y + s^*}{t}, & t_0 \leq \tau \leq t,
\end{cases}
\]

\[
\gamma(\tau) = 
\begin{cases}
\left( \frac{x(t - \tau + as^\ast) - \tau (ac-b)s^\ast}{t + as^\ast}, \quad \frac{y(t-\tau) - \tau s^*}{t} \right), & 0 \leq \tau < t_0, \\
\\
\left( x - \frac{yt \big( x + (ac-b)s^* \big)}{(y + s^*)(t + as^*)} + \tau\,\frac{ty(ac-b) - x\left( t+as^\ast+ay \right)}{t(t + as^*)}, \quad 0 \right), & t_0 \leq \tau \leq t,
\end{cases}
\]

where \( t_0 = \frac{y t}{y + s^\ast} \) and \( x_0 = x - \frac{yt \big( x + (ac-b)s^* \big)}{(y + s^*)(t + as^*)} \). 

Verification is straightforward with respect to establishing that \((\gamma, \eta, l) \in \mathbb{Y}^{\mathbf{x}, t}\) and \(\gamma(t) = (0, 0)\). 

Subsequently, a direct calculation confirms the following:
\[
\int_0^t L(-\eta(\tau)) + F(\gamma,\eta,l)(\tau) d\tau = \frac{\left(-x + b\,s^\ast + c\,t\right)^2}{4(t+as^\ast)} + \frac{(y+s^\ast)^2}{4t} - t.
\]

Consequently, the control triplets selected are optimal. 

\end{proof}

\begin{remark}[\textbf{Optimal Path Existence and Construction}]

We omit a general proof of existence for optimal paths, as it can be adapted from \cite[Theorem 4.7]{ref7}. Crucially, establishing existence alone does not yield formula (3.4). Instead, we construct control triplets achieving the variational lower bound, simultaneously identifying both optimal paths and infimum.  

For the symmetry of \(\varphi^*\), assume \((x,y)\) lies in the first quadrant. Equality in the variational formula requires:
\[
\begin{cases}
K_1(\dot{\gamma}_1(\tau), l(\tau)) = K_1\left(-\frac{x}{t}, \frac{s^*}{t}\right) + \nabla K_1\left(-\frac{x}{t}, \frac{s^*}{t}\right) \cdot \left[(\dot{\gamma}_1(\tau), l(\tau)) - \left(-\frac{x}{t}, \frac{s^*}{t}\right)\right], \\
\\
K_2(\dot{\gamma}_2(\tau), l(\tau)) = K_2\left(-\frac{y}{t}, \frac{s^*}{t}\right) + \nabla K_2\left(-\frac{y}{t}, \frac{s^*}{t}\right) \cdot \left[(\dot{\gamma}_2(\tau), l(\tau)) - \left(-\frac{y}{t}, \frac{s^*}{t}\right)\right].
\end{cases}
\]

i.e.:

\[
\begin{cases}
\left( \dot{\gamma}_1(\tau) + \frac{x}{t}, \; l(\tau) - \frac{s^*}{t} \right) \mathbf{H}\left(K_1(-\frac{x}{t}, \frac{s^*}{t})\right) \begin{pmatrix} \dot{\gamma}_1(\tau) + \frac{x}{t} \\ l(\tau) - \frac{s^*}{t} \end{pmatrix} = 0, \\
\\
\left( \dot{\gamma}_2(\tau) + \frac{y}{t}, \; l(\tau) - \frac{s^*}{t} \right) \mathbf{H}\left(K_2(-\frac{y}{t}, \frac{s^*}{t})\right) \begin{pmatrix} \dot{\gamma}_2(\tau) + \frac{y}{t} \\ l(\tau) - \frac{s^*}{t} \end{pmatrix} = 0.
\end{cases}
\]

where \(\mathbf{H}(K_1)\) and \(\mathbf{H}(K_2)\) are the Hessian matrices of \(K_1\) and \(K_2\).

Under conditions (H1) and (H2), we have
\begin{itemize}

\item \(l(\tau) = 0\) in the field: \(\dot{\gamma}(\tau) = \left( -\frac{x + (ac-b)s^*}{t + as^*}, \quad -\frac{y + s^*}{t} \right)\)

\item \(\dot{\gamma}_2(\tau) = 0\) on the road: \(l(\tau) = \frac{y + s^*}{t},\quad \dot{\gamma}_1(\tau) = \frac{ty(ac-b) - x\left( t+as^\ast+ay \right)}{t(t + as^*)}\)

\end{itemize}

At this juncture, one may derive the optimal triplets via elementary computations. (These optimal control triplets may be construed as the "most efficacious" trajectory from the initial point \( \mathbf{x} = (x,y) \) to the origin \((0, 0)\) within temporal constraint \( t \).) We observe that the optimal path connects the origin and the point \( \mathbf{x} \) in two distinct modalities:
\begin{itemize}
    \item When \( y \geq \dfrac{a}{2t}(x-ct)^{2} + b(x-ct) \), the optimal path constitutes a \emph{rectilinear segment} from \( \mathbf{x} \) to the origin traversed at uniform velocity.
    
    \item When \( y < \dfrac{a}{2t}(x-ct)^{2} + b(x-ct) \), the optimal path initially proceeds linearly through the field, oriented parallel to the normal vector of the contour line of \(\varphi^*(\cdot,t)\), attaining \((x_0, 0)\) at time \( t_0 \). Subsequently, it propagates along the \( x \)-axis to the origin.
\end{itemize}

The critical transitional boundary is delineated by the equation:
\[
y = \dfrac{a}{2t}(x-ct)^{2} + b(x-ct).
\]

\end{remark}

\medskip

\begin{theorem}
The functional \( J \) satisfies Freidlin's condition:
\[
J(\mathbf{x},t) = \inf_{\substack{(\gamma, \eta, l) \in \mathbb{Y}^{\mathbf{x},t} \\ \gamma(t) = \mathbf{0} \\ (\gamma(\tau), t - \tau) \in P}} \left\{ \int_0^t L(-\eta(\tau)) + F(\gamma, \eta, l)(\tau)  d\tau \right\}
\]

for any \((\mathbf{x}, t) \in \partial P\), where \( P := \{(\mathbf{x}, t) : J(\mathbf{x}, t) > 0\} \). Consequently, we have \( v = \max\{0, J\} \).
\end{theorem}

\begin{proof}

The demonstration that \( v = \max\{0, J\} \) under satisfaction of Freidlin's condition is analogous to \cite[Theorem 5.1]{ref2}; ergo, we omit its recapitulation and concentrate exclusively upon the first assertion.

By Lemma 3, we ascertain \( P = \{(\mathbf{x}, t) : \varphi^*(\mathbf{x}, t) > t\} \). For an arbitrary \((\mathbf{x}, t) \in \partial P\), let \(\bar{s} := s^*(\mathbf{x}, t)\), whereupon
\[
\varphi^*(\mathbf{x}, t) = \frac{\left(-x + b\,\bar{s} + c\,t \right)^2}{4(t+a\bar{s})} + \frac{(y+\bar{s})^2}{4t} = t.
\]

We claim that the optimal triplets \((\gamma, \eta, l)\) satisfy \(( \gamma(\tau), t - \tau) \in \overline{P}\), i.e.
\[
\varphi^*(\gamma(\tau), t - \tau) \geq t - \tau.
\]

Consider two distinct regimes:

\textbf{Regime I:} \( y \geq \dfrac{a}{2t}(x-ct)^{2} + b(x-ct) \).\\ 
\\
We have \(s^\ast(x, y, t)=0\) and the optimal path constitutes \(\gamma(\tau) = \left(x(1 - \tau/t), y(1 - \tau/t)\right)\). Hence,
\[
\varphi^*(\gamma(\tau), t - \tau) = \frac{t - \tau}{t} \varphi^*(x, y, t) = t - \tau.
\]

\textbf{Regime II:} \( y < \dfrac{a}{2t}(x-ct)^{2} + b(x-ct) \).\\
\\
Direct calculation reveals that for \( 0 \leq \tau \leq t_0 \) (with \( t_0 \) as defined in the \textbf{Proof of Lemma 3}),
\[
2(\gamma_2(\tau) + \bar{s})(t - \tau + a\bar{s})^2 = a\gamma_1^2(\tau)(t - \tau) - 2(ac-b)\gamma_1(\tau)(t-\tau)^2 + c(ac-2b)(t-\tau)^3.
\]

Consequently, \( a\gamma_1^2(\tau) - 2(ac-b)\gamma_1(\tau)(t-\tau) + c(ac-2b)(t-\tau)^2 \leq 2\gamma_2(\tau)(t - \tau) \) and \( s^*(\gamma(\tau), t - \tau) = \bar{s} \) for \( \tau \in [0, t_0] \). 

Define
\[
\tilde{s}(\tau) := s^* \left( \frac{\gamma(\tau)}{t - \tau}, 1 \right) = \frac{\bar{s}}{t - \tau}.
\]

For \(\tau \in [0,t_0]\), the subsequent derivation obtains:
\begin{align*}
\varphi^*(\gamma(\tau), t - \tau) 
&= \frac{\big[-\gamma_1(\tau)+b\bar{s}+c(t-\tau)\big]^2}{4(t - \tau + a\bar{s})} + \frac{(\gamma_2(\tau) + \bar{s})^2}{4(t - \tau)} \\
&= \frac{(x-b\bar{s}-ct)^2(t - \tau + a\bar{s})}{4(t + a\bar{s})^2} + \frac{(y + \bar{s})^2(t - \tau)}{4t^2}\\
&= (t - \tau) \left( \frac{(x-b\bar{s}-ct)^2(1 + a\tilde{s})}{4(t + a\bar{s})^2} + \frac{(y + \bar{s})^2}{4t^2} \right)\\
&= \frac{t - \tau}{t} \left( \frac{(x-b\bar{s}-ct)^2}{4(t + a\bar{s})} \frac{t}{t + a\bar{s}} + \frac{(y + \bar{s})^2}{4t} + \frac{a(x-b\bar{s}-ct)^2 \tilde{s}t}{4(t + a\bar{s})^2} \right)\\
&= \frac{t - \tau}{t} \left( t - \frac{(x - b\bar{s}-ct)^2 a\bar{s}}{4(t + a\bar{s})^2} + \frac{a(x - b\bar{s}-ct)^2 \tilde{s} t}{4(t + a\bar{s})^2} \right)\\
&= t - \tau + \dfrac{a(x-b\bar{s}-ct)^2\tau\bar{s}}{4t (t + a\bar{s})^{2}} \geq t - \tau.
\end{align*}

At \(\tau = t_0\), we have \(\gamma(t_0) = (x_0, 0)\) and
\[
\varphi^*(x_0, 0, t - t_0) \geq t - t_0,
\]

implying for \(\tau \in [t_0, t)\),
\[
\varphi^*(\gamma(\tau), t - \tau) = (t - \tau) \varphi^*\left(\frac{x_0}{t - t_0}, 0, 1\right) \geq t - \tau.
\]

Conjoining these results, for \(y < \dfrac{a}{2t}(x-ct)^{2} + b(x-ct)\) and \(y \neq 0\), the optimal triplets \((\gamma, \eta, l)\) \(\in\) \(P\). Thus,
\[
J(\mathbf{x}, t) = \inf_{\substack{(\gamma, \eta, l) \in \mathbb{Y}^{\mathbf{x},t} \\ \gamma(t) = 0 \\ (\gamma(\tau), t - \tau) \in P}} \left\{ \int_0^t L(-\eta(\tau)) + F(\gamma, \eta, l)(\tau) d\tau \right\}.
\]

When \(y \geq \dfrac{a}{2t}(x-ct)^{2} + b(x-ct)\) or \(y = 0\), the optimal path satisfies \((\gamma(\tau), t - \tau) \in \partial P\). For the specific case \(x > 0\) and \(y = 0\), the optimal triplet constitutes
\[
\gamma(\tau) = (x(1 - \tau/t), 0), \quad \eta(\tau) = \left(-\frac{x + (ac-b)s^*}{t + as^*}, -\frac{s^*}{t}\right), \quad l(\tau) = \frac{s^*}{t}.
\]

Consider the sequence of control triplets:
\[
\eta^\epsilon(\tau) = 
\begin{cases}
\left( -\frac{x - \epsilon + (ac-b)s^*}{t + as^*}, -\dfrac{s^*}{t} \right), & 0 \leq \tau < t/2, \\
\left( -\frac{x + \epsilon + (ac-b)s^*}{t + as^*}, -\dfrac{s^*}{t} \right), & t/2 \leq \tau \leq t,
\end{cases}
\]

\[
l(\tau) = \frac{s^*}{t}, \quad 0 \leq \tau \leq t,
\]

\[
\gamma^\epsilon(\tau) = \left( x + \int_0^\tau (1+al)\eta_1^{\epsilon} + (ac-b)l,\quad 0 \right), \quad 0 \leq \tau \leq t.
\]

It is manifest that \( (\gamma^\epsilon, \eta^\epsilon, l) \in \mathbb{Y}^{\mathbf{x},t} \) and  
\[
\gamma_1^\epsilon (\tau) > \gamma_1 (\tau), \quad 0 < \tau < t, \quad \gamma^\epsilon (t) = (0, 0).
\]

Given the strict monotonicity of \(\varphi^*\) in \(x\), \( \varphi^* (\gamma^\epsilon (\tau), t - \tau) > t - \tau \), i.e., \( (\gamma^\epsilon (\tau), t - \tau) \in P \). Furthermore, direct computation yields  
\[
\int_0^t L(-\eta^\epsilon (\tau)) + F(\gamma^\epsilon, \eta^\epsilon, l)  d\tau \to J(\mathbf{x}, t) \quad \text{as } \epsilon \to 0,
\]

establishing Freidlin's condition for \( y = 0 \). The case \( y \geq \dfrac{a}{2t}(x-ct)^{2} + b(x-ct) \) admits an analogous treatment.

\end{proof}

\section{CONICAL DOMAINS}

We discuss a novel and natural extension of our approach to the case in a general conical domain. For a fixed $\alpha\in(0,\tfrac{\pi}{2}]$, define
\[
\Omega_{\alpha}=\{(r\sin\vartheta,\; r\cos\vartheta):\ r>0,\ \vartheta\in(\tfrac{\pi}{2}-2\alpha,\tfrac{\pi}{2})\}.
\]

In this setting there are two portions of the road:
\[
\Gamma_0=(0,\infty)\times\{0\}
\qquad\text{and}\qquad
\Gamma_\alpha=\{(x,y):\ y>0,\ \frac{x}{y}=\tan\!\big(\tfrac{\pi}{2}-2\alpha\big)\}.
\]

To formulate the boundary condition on the slanted ray \(\Gamma_\alpha\), it is convenient to introduce the unit tangent and unit outward normal on \(\Gamma_\alpha\). For the geometry above one may take
\[
\mathbf{\tau_\alpha}=(\cos 2\alpha,\;\sin 2\alpha),\qquad
\mathbf{n_\alpha}=(-\sin 2\alpha,\;\cos 2\alpha),
\]

so that, on \(\Gamma_\alpha\),
\[
\partial_\mathbf{s} := \mathbf{\tau_\alpha} \cdot \nabla = \cos(2\alpha)\,\partial_x+\sin(2\alpha)\,\partial_y,
\qquad
\partial_\mathbf{n} := \mathbf{n_\alpha} \cdot \nabla = -\sin(2\alpha)\,\partial_x+\cos(2\alpha)\,\partial_y.
\]

Thus, system (1.2) takes the following form:
\begin{equation}
\begin{cases}
u_t - \Delta u = u(1-u), & \quad (t,x,y) \in (0,\infty) \times \Omega_\alpha, \\[4pt]
a\,u_{xx} - c\,u_x + u_y = 0, & \quad (t,x,y) \in (0,\infty) \times \Gamma_0, \\[4pt]
\widetilde{a}\,\partial_\mathbf{s}^{2} u - \widetilde{c}\,\partial_\mathbf{s} u - \partial_\mathbf{n} u = 0, & \quad (t,x,y) \in (0,\infty) \times \Gamma_\alpha, \\[4pt]
u(x,y,0)=u_0(x,y), &\quad (x,y) \in \overline{\Omega_\alpha}.
\end{cases}
\tag{4.1}
\end{equation}

\(\text{(H1)}\) and \(\text{(H2)}\) in Section 2 ought to be revised so as to comply with the following format:
\[
\begin{cases}
\gamma(0) = (x_\mathbf{s},\,x_\mathbf{n}), \\
l(\tau) \geq 0 \text{ for a.e. } \tau \in [0, t], \\
l(\tau) = 0 \text{ if } \gamma_{\mathbf{n}}(\tau) \neq 0 \text{ for a.e. } \tau \in [0, t],
\end{cases}
\tag{$\mathcal{H}1$}
\]

and
\[
\begin{cases}
\text{the function} \quad \tau \mapsto F(\gamma, \eta, l)(\tau) := l(\tau) G\left( \frac{\eta_\mathbf{s} - \dot{\gamma}_\mathbf{s}}{l} (\tau) \right) \quad \text{is integrable on } [0, t], \\
\dot{\gamma}(\tau) = \eta(\tau) \text{ if } l(\tau) = 0 \text{ for a.e. } \tau \in [0, t], \\
\eta_\mathbf{n}(\tau) = \dot{\gamma}_\mathbf{n}(\tau) - l(\tau) \text{ if } l(\tau) \neq 0 \text{ for a.e. } \tau \in [0, t].
\end{cases}
\tag{$\mathcal{H}2$}
\]

Here, the expression \(l(\tau) G\left( \frac{\eta_\mathbf{s} - \dot{\gamma}_\mathbf{s}}{l} (\tau) \right)\) in $(\mathcal{H}2)$ is defined only on \(\{\tau \in [0, t] : l(\tau) > 0\}\), but we understand that
\[
l(\tau) G\left( \frac{\eta_\mathbf{s} - \dot{\gamma}_\mathbf{s}}{l} (\tau) \right) = 
\begin{cases}
\frac{\left(\eta_\mathbf{s} - \dot{\gamma}_\mathbf{s} - b l\right)^2 (\tau)}{4a l(\tau)} & \text{if } l(\tau) > 0, \\
\\
0 & \text{if } l(\tau) = 0.
\end{cases}
\]

Denote
\[
\widehat{\mathbb{Y}^{\mathbf{x},t}} := \left\{(\gamma, \eta, l) : (\gamma, \eta, l) \text{ satisfies } (\mathcal{H}1) \text{ and } (\mathcal{H}2)\right\}.
\]

We now consider the scenario where the diffusion and drift coefficient on the roads differ, as previously indicated in Equation (4.1). Specifically, we impose the condition:
\[
\widetilde{a} > a > 0\;, \quad \widetilde{c} > c > 0\;.
\]

For any $\widetilde{a} > a > 0\;,\;\widetilde{c} > c > 0\;$, we have \( w_\alpha^{(a, \widetilde{a}, c, \widetilde{c})} \) denotes the solution to the equation:
\[
\min\big\{\partial_{t}\,w_\alpha^{(a, \widetilde{a}, c, \widetilde{c})} + |\nabla w_\alpha^{(a, \widetilde{a}, c, \widetilde{c})}|^2 + 1,\;w_\alpha^{(a, \widetilde{a}, c, \widetilde{c})}\big\}=0 \qquad \text{in}\;\;\mathbb{R_+}\,\times\,\Omega_\alpha, 
\tag{4.2}
\]

and
\begin{align}
a\big(\partial_x w_\alpha^{(a, \widetilde{a}, c, \widetilde{c})}\big)^2 - \partial_y w_\alpha^{(a, \widetilde{a}, c, \widetilde{c})} + c\,\partial_x w_\alpha^{(a, \widetilde{a}, c, \widetilde{c})} &= 0, \qquad \text{on}\;\;\Gamma_0 \tag{4.3} \\
\widetilde{a}\big(\partial_{\textbf{s}} w_\alpha^{(a, \widetilde{a}, c, \widetilde{c})}\big)^2 + \partial_{\mathbf{n}} w_\alpha^{(a, \widetilde{a}, c, \widetilde{c})} + \widetilde{c}\,\partial_{\mathbf{s}} w_\alpha^{(a, \widetilde{a}, c, \widetilde{c})} &= 0, \qquad \text{on}\;\;\Gamma_\alpha \tag{4.4}
\end{align}

Let \(u\) be any solution of \((4.1)\), and let \(u^\varepsilon,\; v^\varepsilon\) be defined analogously as in Section 1. Then
\[
v^\varepsilon\; \longrightarrow\; w_\alpha^{(a, \widetilde{a}, c, \widetilde{c})}
\]
locally uniformly. 

\medskip

Throughout this work we restrict attention to the case of \emph{monotonic variation}. Under this restriction, an increase in \(a\), an increase in \(c\), or a simultaneous increase in both parameters have equivalent effects on the monotonic quantity \(w_\alpha^{(a, \widetilde{a}, c, \widetilde{c})}\) since all Hamiltonians are increasing in both \(a\) and \(c\). Consequently, it suffices to analyse the representative scenario in which \(a\) increases while \(c\) is held fixed. The remaining cases admit analogous analysis and are therefore omitted here.

\medskip

\begin{lemma}
Fix $\alpha\in(0,\tfrac{\pi}{2})$. Then, for every $(t,x,y)\in(0,\infty)\times\Omega_\alpha$,
\[
w_\alpha^{(a, \widetilde{a}, \widetilde{c}, \widetilde{c})}(t,x,y)
\;\geq\; w_\alpha^{(\widetilde{a}, \widetilde{a}, \widetilde{c}, \widetilde{c})}(t,x,y)
\;=\; \max\Big\{0,\; J_\alpha^{(\widetilde{a}, \widetilde{a}, \widetilde{c}, \widetilde{c})}\Big\}
\]

where
\[
J_\alpha^{(\widetilde{a}, \widetilde{a}, \widetilde{c}, \widetilde{c})}\,=\,\inf_{(\gamma,\eta,l)\in \widehat{\mathbb{Y}^{\mathbf{x},t}}} \left\{ \int_0^t L(-\eta(\tau)) + F(\gamma,\eta,l)(\tau) d\tau \,\bigg|\, \gamma(t) = (0,0) \right\},
\]

and, moreover,
\[
J_\alpha^{(\widetilde{a}, \widetilde{a}, \widetilde{c}, \widetilde{c})}(t,x,y) = \min\big\{J^{(\widetilde{a}, \widetilde{a}, \widetilde{c}, \widetilde{c})}(t,x,y),\; J^{(\widetilde{a}, \widetilde{a}, \widetilde{c}, \widetilde{c})}(t,\Psi_\alpha(x,y))\big\}.
\]
\end{lemma}

\begin{proof}

\medskip

The first inequality follows directly from the monotonicity of the Hamiltonians with respect to the road-diffusion parameter $\widetilde{a}$. More precisely, since each Hamiltonian appearing in the Hamilton--Jacobi formulation is nondecreasing in $\widetilde{a}$, the function $w_\alpha^{(\widetilde{a}, \widetilde{a}, \widetilde{c}, \widetilde{c})}$ is a viscosity subsolution of the equation satisfied by $w_\alpha^{(a, \widetilde{a}, \widetilde{c}, \widetilde{c})}$, and the comparison principle (or the variational characterization) yields
\[
w_\alpha^{(a, \widetilde{a}, \widetilde{c}, \widetilde{c})} \ge w_\alpha^{(\widetilde{a}, \widetilde{a}, \widetilde{c}, \widetilde{c})}.
\]

\medskip

By the rotational monotonicity of \(J\) (refer to Appendix C), it follows that on the subdomain $\Omega_{\alpha/2}$ one has
\[
J_\alpha^{(\widetilde{a}, \widetilde{a}, \widetilde{c}, \widetilde{c})} = J^{(\widetilde{a}, \widetilde{a}, \widetilde{c}, \widetilde{c})} \qquad\text{and}\qquad \widetilde J_\alpha^{(\widetilde{a}, \widetilde{a}, \widetilde{c}, \widetilde{c})} = \widetilde J^{(\widetilde{a}, \widetilde{a}, \widetilde{c}, \widetilde{c})}
\]

where \(\widetilde{J}\,=\,J \circ \Psi_{\alpha}\).

Hence, away from the interface line $\Gamma_{\alpha/2}$, the function $J_\alpha^{(\widetilde{a}, \widetilde{a}, \widetilde{c}, \widetilde{c})}$ satisfies, in the viscosity sense,
\begin{equation}
\partial_{t}\,J_\alpha^{(\widetilde{a}, \widetilde{a}, \widetilde{c}, \widetilde{c})} + |\nabla J_\alpha^{(\widetilde{a}, \widetilde{a}, \widetilde{c}, \widetilde{c})}|^2 + 1=0 \qquad\text{in}\;\;\mathbb R_+\times\Omega_\alpha,
\tag{4.5}
\end{equation}

and, on the roads, the pair $(J_\alpha^{(\widetilde{a}, \widetilde{a}, \widetilde{c}, \widetilde{c})},\widetilde J_\alpha^{(\widetilde{a}, \widetilde{a}, \widetilde{c}, \widetilde{c})})$ satisfies in the strong sense the boundary system
\begin{align}
\widetilde{a}\big(\partial_x J_\alpha^{(\widetilde{a}, \widetilde{a}, \widetilde{c}, \widetilde{c})}\big)^2
-\partial_y J_\alpha^{(\widetilde{a}, \widetilde{a}, \widetilde{c}, \widetilde{c})}
+ \widetilde{c}\,\partial_x J_\alpha^{(\widetilde{a}, \widetilde{a}, \widetilde{c}, \widetilde{c})}
&=0, &&\text{on }\Gamma_0 \tag{4.6}\\[4pt]
\widetilde a\big(\partial_{\mathbf{s}} \widetilde J_\alpha^{(\widetilde{a}, \widetilde{a}, \widetilde{c}, \widetilde{c})}\big)^2
+\partial_{\mathbf{n}} \widetilde J_\alpha^{(\widetilde{a}, \widetilde{a}, \widetilde{c}, \widetilde{c})}
+ \widetilde c\,\partial_{\mathbf{s}} \widetilde J_\alpha^{(\widetilde{a}, \widetilde{a}, \widetilde{c}, \widetilde{c})}
&=0, &&\text{on }\Gamma_\alpha. \tag{4.7}
\end{align}

Moreover, $J_\alpha^{(\widetilde{a}, \widetilde{a}, \widetilde{c}, \widetilde{c})}$ is Lipschitz continuous on compact subsets, and therefore it satisfies (4.5) in the classical sense almost everywhere in $\mathbb R_+\times\Omega_\alpha$. 

To promote this local statement to a global viscosity statement, we proceed by approximation. Concretely, let $\rho_\varepsilon$ be a standard space–time mollifier and set
\[
J^{(\widetilde{a}, \widetilde{a}, \widetilde{c}, \widetilde{c}),\varepsilon} := J^{(\widetilde{a}, \widetilde{a}, \widetilde{c}, \widetilde{c})} * \rho_\varepsilon,\qquad \widetilde J^{(\widetilde{a}, \widetilde{a}, \widetilde{c}, \widetilde{c}),\varepsilon} := (J^{(\widetilde{a}, \widetilde{a}, \widetilde{c}, \widetilde{c})}\circ\Psi_\alpha)*\rho_\varepsilon.
\]

Due to the convexity of the Hamiltonian with respect to the gradient, Jensen's inequality implies that $J^{(\widetilde{a}, \widetilde{a}, \widetilde{c}, \widetilde{c}),\varepsilon}$ and $\widetilde J^{(\widetilde{a}, \widetilde{a}, \widetilde{c}, \widetilde{c}),\varepsilon}$ pointwise satisfy
\begin{align*}
\begin{cases}
\partial_t J^{(\widetilde{a}, \widetilde{a}, \widetilde{c}, \widetilde{c}),\varepsilon} + |\nabla J^{(\widetilde{a}, \widetilde{a}, \widetilde{c}, \widetilde{c}),\varepsilon}|^2 + 1 \le 0 \\
\partial_t \widetilde J^{(\widetilde{a}, \widetilde{a}, \widetilde{c}, \widetilde{c}),\varepsilon} + |\nabla \widetilde J^{(\widetilde{a}, \widetilde{a}, \widetilde{c}, \widetilde{c}),\varepsilon}|^2 + 1 \le 0
\end{cases}
\end{align*}

on any fixed compact subset of $\mathbb R_+\times\Omega_\alpha$. Taking the pointwise minimum
\[
v^\varepsilon:=\min\{J^{(\widetilde{a}, \widetilde{a}, \widetilde{c}, \widetilde{c}),\varepsilon},\widetilde J^{(\widetilde{a}, \widetilde{a}, \widetilde{c}, \widetilde{c}),\varepsilon}\}
\]

yields a family of continuous viscosity subsolutions which converges uniformly on compacts to $J_\alpha^{(\widetilde{a}, \widetilde{a}, \widetilde{c}, \widetilde{c})}$. By the stability of viscosity subsolutions under uniform limits (see \cite[Chapter II, Proposition 5.1]{ref12}), the limit $J_\alpha^{(\widetilde{a}, \widetilde{a}, \widetilde{c}, \widetilde{c})}$ is itself a viscosity subsolution of (4.5) on the whole domain $\mathbb R_+\times\Omega_\alpha$. Thus $J_\alpha^{(\widetilde{a}, \widetilde{a}, \widetilde{c}, \widetilde{c})}$ is a strong subsolution of the system (4.5)-(4.6)-(4.7).

\medskip

On the other hand, since $J_\alpha^{(\widetilde{a}, \widetilde{a}, \widetilde{c}, \widetilde{c})}$ is the pointwise minimum of two (strong) solutions, similar to the argument in Theorem 1 \& 2 (refer to Section 2 \& 3), we have $J_\alpha^{(\widetilde{a}, \widetilde{a}, \widetilde{c}, \widetilde{c})}$ is a strong supersolution of (4.5)-(4.6)-(4.7). 

Combining the subsolution and supersolution properties yields that $J_\alpha^{(\widetilde{a}, \widetilde{a}, \widetilde{c}, \widetilde{c})}$ is a strong (viscosity) solution of (4.5)-(4.6)-(4.7).

Finally, similar to the argument in Theorem 4 and by \cite[Theorem 5.1]{ref2}, \( \max\Big\{0,\;J_\alpha^{(\widetilde{a}, \widetilde{a}, \widetilde{c}, \widetilde{c})}\Big\} \) is a strong solution of (4.2)-(4.3)-(4.4) in $\Omega_\alpha$. This completes the proof.

\end{proof}

\begin{theorem}
Let \( \alpha \in (0, \frac{\pi}{2}] \). Suppose
\[
J^{(\widetilde{a}, \widetilde{a}, \widetilde{c}, \widetilde{c})}(1, r \sin(\frac{\pi}{2} - 2\alpha), r \cos(\frac{\pi}{2} - 2\alpha)) \leq J^{(a, a, \widetilde{c}, \widetilde{c})}(1, r, 0) \quad \text{for all \,} r \geq 0,
\]

then
\[
w_\alpha^{(a, \widetilde{a}, \widetilde{c}, \widetilde{c})}(t,x,y) = \max \Big\{ 0, \min \big\{ J^{(a, a, \widetilde{c}, \widetilde{c})}(t, x, y),\;J^{(\widetilde{a}, \widetilde{a}, \widetilde{c}, \widetilde{c})}(t, \Psi_\alpha(x, y)) \big\} \Big\} \quad \text{in \,} (0, \infty) \times \overline{\Rplus^2}.
\]

\end{theorem}

\begin{proof}
Define
\[
\widehat{J}(t, x, y) = \min \big\{ J^{(a, a, \widetilde{c}, \widetilde{c})}(t, x, y),\;J^{(\widetilde{a}, \widetilde{a}, \widetilde{c}, \widetilde{c})}(t, \Psi_\alpha(x, y)) \big\}.
\]

First, following the proof of Lemma 3, we observe that \(\widehat{J}\) is a strong solution of the Hamilton-Jacobi equation (4.2)-(4.3)-(4.4) without obstacle. Then, following the proof of \cite[Theorem 5.1]{ref2}, we conclude that
\[
w_\alpha^{(a, \widetilde{a}, \widetilde{c}, \widetilde{c})}(t,x,y) = \max \left\{ 0, \widehat{J} \right\}.
\]

\end{proof}

\section{NUMERICAL SIMULATION}

This section provides numerical simulations of the level set
\[
\Omega:=\{(x,y)\in\overline{\mathbb{R}_+^2}:\;\varphi^{\ast}(x,y,1) < 1\}.
\]

\begin{figure}[H]
  \centering
  \includegraphics[width=\textwidth]{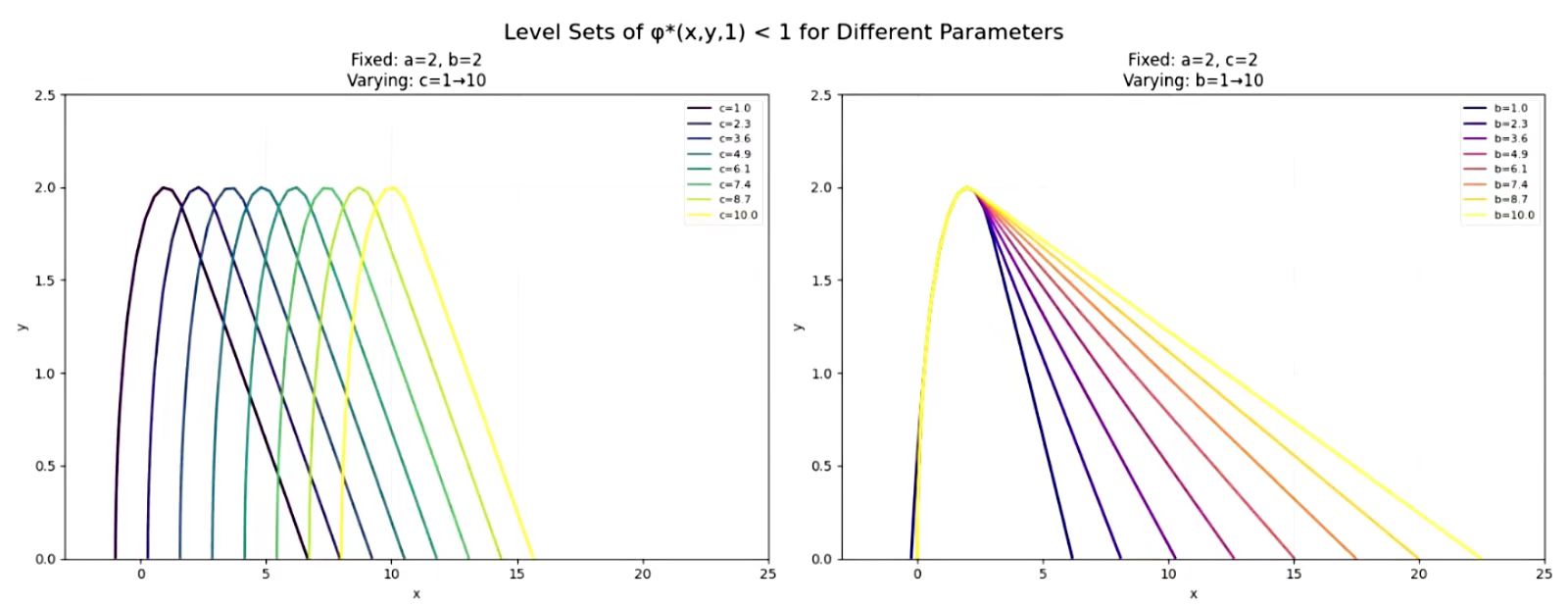}
  \caption{\textbf{Level sets of \(\varphi^*(x,y,1)<1\) for different parameter sweeps. }The panel shows (left) \(a=2,\;b=2\) with \(c\) varying from \(1\) to \(10\) and (right) \(a=2,\;c=2\) with \(b\) varying from \(1\) to \(10\).}
\end{figure}

We produce a panel to illustrate two parameter experiments:
\begin{enumerate}
  \item \textbf{(left)} Fix \(a=2\), \(b=2\); vary \(c\) in \(\{1,2,\dots,10\}\).
  \item \textbf{(right)} Fix \(a=2\), \(c=2\); vary \(b\) in \(\{1,2,\dots,10\}\).
\end{enumerate}

The plotted results illustrate the following qualitative behaviors:
\begin{itemize}
  \item Increasing \(c\) (left) shifts the invaded region downstream and deforms the contour in the \(x\)-direction, as field advection favors movement in the \((+x)\) direction.
  
  \item Increasing \(b\) (right) enhances travel along the road — the invaded region extends further along the road direction and the critical transition between "field-dominated" and "road-assisted" propagation moves.
\end{itemize}

\bigskip

\appendix
\section*{APPENDIX}

\section{Asymptotic Derivation of the Wentzell-Type Boundary Condition}

We consider a two-layer field-road geometry in \(\overline{\Rplus^2}\): the road is the thin strip $\{(x,y):0<y<\delta\}$ of thickness $\delta>0$, and the field occupies $\{y>\delta\}$. Let $u(x,y,t)$ denote the population density. The goal of this section is to derive, in the singular limit $\delta\downarrow0$, the effective boundary condition on $y=0$ of Wentzell-type that couples the field to the concentrated transport along the road.

\subsection{Model setup and notation}

In the field ($y>\delta$) we take unit diffusion and a tangential advection $c$ along the $x$-direction. Inside the road ($0<y<\delta$) the longitudinal diffusion (along-$x$) is large and denoted by $\sigma=\sigma(\delta)>0$, and the tangential advection in the road is denoted by $\widetilde b(\delta)$, which may depend on $\delta$. The reaction term is the Fisher-KPP nonlinearity $f(u)=u(1-u)$ (the argument below does not rely on its precise form beyond being $\mathcal O(1)$). Thus, the governing equations are
\begin{equation}\label{eq:full_system}
\begin{cases}
u_t - \Delta u + c\,u_x = u(1-u), & y>\delta,\\[4pt]
u_t - \sigma(\delta)\,(u_{xx}+u_{yy}) + \widetilde b(\delta)\,u_x = u(1-u), & 0<y<\delta.
\end{cases}
\end{equation}

We write the advective-diffusive flux as
\[
\mathbf J = -D\,\nabla u + (v_{\parallel}\,u,\;0),
\]
with $(D,\,v_{\parallel})=(1,\,c)$ in the field and $(D,\,v_{\parallel})=(\sigma(\delta),\,\widetilde b(\delta))$ in the road. The upward unit normal at the interface $G_\delta:=\{y=\delta\}$ is $\mathbf{n}=(0,\,1)$.

\subsection{Interface conditions}

We impose continuity of density across the microscopic interface,
\begin{equation}\label{eq:continuity}
u(x,\delta^+,t)=u(x,\delta^-,t),
\end{equation}

and the impermeability (no normal flux) at the road bottom:
\begin{equation}\label{eq:bottom_no_flux}
\partial_y u(x,0^+,t)=0.
\end{equation}

Compute the net normal flux across $G_\delta$ from the field and road sides:
\[
F^+ = \int_{G_\delta} \mathbf J(x,\delta^+,t)\cdot(0,\,-1)\,dx
       = \int_{G_\delta} \partial_y u(x,\delta^+,t)\,dS,
\]

\[
F^- = \int_{G_\delta} \mathbf J(x,\delta^-,t)\cdot(0,\,1)\,dx
       = \int_{G_\delta} -\sigma(\delta)\,\partial_y u(x,\delta^-,t)\,dS.
\]

Requiring flux continuity $F^+\,+\,F^-\,=\,0$ yields the pointwise transmission condition
\begin{equation}\label{eq:transmission}
\partial_y u(x,\delta^+,t)=\sigma(\delta)\,\partial_y u(x,\delta^-,t).
\end{equation}

\subsection{Road averaging and integrated equation}
Introduce the road-averaged density
\begin{equation}\label{eq:road_avg}
v(x,t):=\frac{1}{\delta}\int_0^\delta u(x,y,t)\,dy.
\end{equation}

Integrate the road equation in \eqref{eq:full_system} over $y\in(0,\delta)$:

\begin{equation}\label{eq:integrated_road}
\int_0^\delta u_t\,dy - \sigma(\delta)\int_0^\delta(u_{xx}+u_{yy})\,dy + \widetilde b(\delta)\int_0^\delta u_x\,dy
= \int_0^\delta u(1-u)\,dy.
\end{equation}

Using $\int_0^\delta u_t\,dy=\delta v_t$ and $\int_0^\delta u_x\,dy=\delta v_x$, and applying the fundamental theorem of calculus to the $y$-second derivative,
\[
\sigma(\delta)\int_0^\delta u_{yy}\,dy = \sigma(\delta)\bigl[\partial_y u(x,\delta^-,t)-\partial_y u(x,0^+,t)\bigr].
\]

With \eqref{eq:bottom_no_flux} and \eqref{eq:transmission} we have
\[
\sigma(\delta)\int_0^\delta u_{yy}\,dy = \sigma(\delta)\,\partial_y u(x,\delta^-,t)=\partial_y u(x,\delta^+,t).
\]

Substituting into \eqref{eq:integrated_road} yields
\begin{equation}\label{eq:before_limit}
\delta v_t - \sigma(\delta)\delta v_{xx} + \widetilde b(\delta)\,\delta v_x - \partial_y u(x,\delta^+,t)
= \int_0^\delta u(1-u)\,dy.
\end{equation}

Rearranged,
\[
-\,\sigma(\delta)\delta\,v_{xx} + \widetilde b(\delta)\,\delta\,v_x - \partial_y u(x,\delta^+,t)
= \int_0^\delta u(1-u)\,dy - \delta v_t. \tag{\(\ast\)}
\]

\subsection{Singular limit and scaling assumptions}

We pass to the thin-road limit $\delta\downarrow0$ under the following physically natural scalings and assumptions:
\begin{enumerate}
  \item \textbf{Integrated longitudinal transport finite:} the product
  \[
    a := \lim_{\delta\downarrow0} \sigma(\delta)\,\delta
  \]
  exists and satisfies $a>0$. The constant $a$ represents the integrated diffusive transport capacity along the vanishing road.
  
  \item \textbf{Concentrated tangential advection:} the scaled advection
  \[
    b := \lim_{\delta\downarrow0} \widetilde b(\delta)\,\delta
  \]
  exists (possibly $b=0$). Equivalently $\widetilde b(\delta)=b/\delta + o(1/\delta)$. If $\widetilde b(\delta)=\mathcal O(1)$ then $b=0$ and no first-order tangential term persists in the limit.
  
  \item \textbf{Negligible storage and reaction in the vanishing layer:}
  \[
    \delta v_t \to 0, \qquad \int_0^\delta u(1-u)\,dy = \mathcal O(\delta) \to 0,
  \]
  as $\delta\downarrow0$. This corresponds to the road being thin enough that its volumetric storage and intrinsic reaction are negligible on the time scales of interest.
  
  \item \textbf{Trace convergence:} the average converges to the trace on the upper boundary,
  \[
    v(x,t)\to u(x,0^+,t)\qquad\text{as }\delta\downarrow0.
  \]
\end{enumerate}

\subsection{Limit and the Wentzell boundary condition}
Let $\delta\downarrow0$, we obtain the pointwise limit at $y=0$:
\[
-\,a\,\partial_{xx} v(x,t) + b\,\partial_x v(x,t) - \partial_y u(x,0^+,t) = 0.
\]

Replacing $v(x,t)$ by its limit $u(x,0^+,t)$ yields the effective boundary condition (Wentzell-type)
\[
a\,u_{xx}(x,0,t) - b\,u_x(x,0,t) + \partial_y u(x,0^+,t) = 0.
\tag{\(\Delta\)}
\]

\section{Strict Convexity of \(\Omega\)}

\medskip

By homogeneity, it suffices to study the normalized profile at time \(t=1\). Define
\[
\Phi(x,y)\;:=\;\varphi^*(x,y,1)
\;=\;\min_{s\ge 0} \; g(x,y,s),
\qquad
g(x,y,s):=\frac{(-x + b s + c)^2}{4(1+ a s)}+\frac{(y+s)^2}{4}.
\]

We consider the level set
\[
\Omega:=\{(x,y)\in\overline{\mathbb{R}_+^2}:\;\Phi(x,y) < 1\}.
\]

\begin{proposition}
The function \(\Phi\) is convex on its domain, and the level set \(\Omega\) is strictly convex.
\end{proposition}

\begin{proof}
We split the proof into three steps.

\medskip

\textbf{Step 1: Joint convexity of \(g(x,y,s)\).}
Write
\[
4g(x,y,s)=h_1(x,s)+h_2(y,s),
\qquad
h_1(x,s)=\frac{(-x+bs+c)^2}{1+as},\quad h_2(y,s)=(y+s)^2.
\]

The function \(h_2\) has Hessian
\[
\nabla^2_{(y,s)} h_2 = \begin{pmatrix}2 & 2\\[4pt] 2 & 2\end{pmatrix},
\]

which is positive semidefinite; hence \(h_2\) is convex in \((y,s)\).

For \(h_1\) define \(u(x,s):=-x + b s + c\) and \(v(s):=1 + a s\). Consider \(k(u,v)=u^2/v\) for \(v>0\). Its Hessian equals
\[
\nabla^2 k(u,v)=\begin{pmatrix} \dfrac{2}{v} & -\dfrac{2u}{v^2}\\[6pt] -\dfrac{2u}{v^2} & \dfrac{2u^2}{v^3}\end{pmatrix},
\]

and for any \(w=(w_1,w_2)^\top\),
\[
w^\top \nabla^2 k(u,v) w = \frac{2}{v}\Bigl( w_1 - \frac{u}{v} w_2\Bigr)^2 \ge 0.
\]

Since \((u,v)\) is an affine function of \((x,s)\), \(h_1(x,s)=k(u(x,s),v(s))\) is convex in \((x,s)\). Therefore \(4g=h_1+h_2\) is convex in \((x,y,s)\), and so is \(g\).

\medskip

\noindent\textbf{Step 2: Convexity of \(\Phi\).}
Let \((x_i,y_i)\), \(i=1,2\), be arbitrary and fix \(t\in(0,1)\). For each \(i\) pick a minimizer \(s_i^*\ge0\) of \(s\mapsto g(x_i,y_i,s)\). Define
\[
(x_t,y_t):=t(x_1,y_1)+(1-t)(x_2,y_2),\qquad s_t:=t s_1^* + (1-t) s_2^*.
\]

By convexity of \(g\),
\[
g(x_t,y_t,s_t)\le t g(x_1,y_1,s_1^*) + (1-t) g(x_2,y_2,s_2^*) = t\Phi(x_1,y_1)+(1-t)\Phi(x_2,y_2).
\]

Since \(\Phi(x_t,y_t)\le g(x_t,y_t,s_t)\) we conclude \(\Phi(x_t,y_t)\le t\Phi(x_1,y_1)+(1-t)\Phi(x_2,y_2)\). Thus \(\Phi\) is convex.

\medskip

\noindent\textbf{Step 3: Strict convexity of \(\Omega\).}
Assume, toward a contradiction, that there exist distinct boundary points \((x_1,y_1)\neq(x_2,y_2)\) with \(\Phi(x_i,y_i)=1\) and some \(t_0\in(0,1)\) such that\(\Phi(x_{t_0},y_{t_0})=1\), where\((x_{t_0},y_{t_0})=t_0(x_1,y_1)+(1-t_0)(x_2,y_2)\).

Let \(s_i^*\) be the minimizer at \((x_i,y_i)\) (so \(g(x_i,y_i,s_i^*)=1\)) and set
\[
s_{t_0}=t_0 s_1^* + (1-t_0) s_2^*.
\]

Convexity of \(g\) yields
\[
1=\Phi(x_{t_0},y_{t_0})\le g(x_{t_0},y_{t_0},s_{t_0})\le t_0 g(x_1,y_1,s_1^*)+(1-t_0)g(x_2,y_2,s_2^*)=1,
\]

hence all inequalities are equalities. Equality in the convexity inequality implies that \(g\) is affine along the straight segment
\[
\gamma(r)=(x(r),y(r),s(r))=r(x_1,y_1,s_1^*)+(1-r)(x_2,y_2,s_2^*),\quad r\in[0,1].
\]

Consequently the second directional derivative of \(g\) in direction \(v=(x_2-x_1,y_2-y_1,s_2^*-s_1^*)\) vanishes for all interior points of the segment:
\[
D^2_v\,g(x(r),y(r),s(r))\equiv 0\quad\text{for } r\in(0,1).
\]

Compute the two contributions to \(D^2_v(4g)\):

(i) For \(h_2(y,s)=(y+s)^2\) we have \(D^2_v\,h_2 = 2(dy + ds)^2\). Thus it must hold on the segment that
\[
dy + ds = 0,
\]

i.e. \(y(r)+s(r)\) is constant along the segment; denote this constant by \(C\ge0\).

(ii) For \(h_1\), with \(u=-x+bs+c,\ v=1+as\), the second directional derivative reads
\[
D^2_v\,h_1 = \frac{2}{v}\Bigl( -dx + b\,ds - \frac{u}{v} a\,ds\Bigr)^2.
\]

Thus \(D^2_v\,h_1\equiv 0\) implies
\[
-\,dx + \Bigl(b - a\frac{u}{v}\Bigr) ds = 0,
\]

which in turn is equivalent to \(d\bigl(\tfrac{u}{v}\bigr)=0\). Hence \(\tfrac{u}{v}\) is constant along the segment; denote this constant by \(k\in\mathbb{R}\). Summarizing:
\[
y(r)+s(r)\equiv C,\qquad \frac{u(x(r),s(r))}{1+a s(r)}\equiv k\quad\text{for all } r\in[0,1].
\]

Now use the first-order optimality condition at an endpoint \((x_i,y_i,s_i^*)\) (where \(s_i^*\) is an interior minimizer):
\[
\frac{\partial g}{\partial s}(x_i,y_i,s_i^*) = 0.
\]

After straightforward differentiation (quotient rule), this condition is equivalent to
\[
\frac{2u b v - a u^2}{v^2} + 2(y+s) = 0,
\]

where \(u=-x+bs+c\) and \(v=1+as\). Substituting \(u=k v\) and \(y+s=C\) yields
\[
2 k b - a k^2 + 2 C = 0,
\]

or equivalently
\[
a k^2 + 2 C = 2 b k.
\]

On the other hand, since \(u=k v\) we have
\[
h_1(x(r),s(r))=\frac{u^2}{v}=\frac{(k v)^2}{v}=k^2 v = k^2(1+a s(r)),
\]

so along the whole segment \(\partial_s h_1 = a k^2\). Inserting this into the one-dimensional first-order condition (which reads \(\tfrac{1}{4}(\partial_s h_1 + 2(y+s))=0\)) gives
\[
k^2 a + 2 C = 0.
\]

Since \(a>0\) and \(C\ge0\), we conclude \(k^2 a = 0\) and \(C=0\). Hence \(k=0\) and \(C=0\). Thus along the entire segment
\[
u\equiv 0,\qquad y\equiv 0,\qquad s\equiv 0.
\]

From \(u\equiv 0\) we deduce \(-x + bs + c \equiv 0\); with \(s\equiv 0\) this gives \(x\equiv c\). Consequently every point on the segment equals \((c,0)\), contradicting the assumption that the two endpoints are distinct. Therefore the assumption was false and \(\Omega\) is strictly convex.

\begin{remark}
In the above argument, we used the first-order condition $\partial_s g(x,y,s^*)=0$ at the endpoints of a hypothetical straight segment lying in the level set, which assumes that the corresponding minimizers $s^*$ at those endpoints are interior ($s^*>0$). We now briefly discuss the remaining boundary cases.

\begin{itemize}
\item If both endpoints have $s^*=0$, then $\Phi(x,y)=g(x,y,0)=\frac{(x-c)^2+y^2}{4}$. The level $\Phi=1$ in this regime is a portion of a circle centered at $(c,0)$ and thus cannot contain a non-degenerate straight segment in the half-plane $y\ge0$ unless the two endpoints coincide (indeed a circle arc is strictly convex). Hence, this case does not produce a non-trivial line segment in the level set boundary.

\item If one endpoint has $s^*=0$ and the other has $s^*>0$, the continuity of $\Phi$ and the uniqueness of minimizers (Proposition 2) imply that along a small neighborhood the minimizer cannot jump from $0$ to a strictly positive value while keeping $g$ affine; a direct contradiction with the equality condition in convexity. Consequently, this mixed case is impossible.
\end{itemize}

Combining these observations with the interior case treated above completes the proof that no non-degenerate straight segment can lie in $\{\Phi=1\}$; hence $\Omega=\{\Phi<1\}$ is strictly convex.
\end{remark}

\end{proof}

\section{Rotational Monotonicity of \(J\)}

\begin{theorem}
\(J\) is rotationally nondecreasing from the positive \(x\)-axis to the positive \(y\)-axis.
\end{theorem}

\begin{proof}
Fix \(r\ge0\) and \(t>0\). Define
\[
\Phi(\theta):=\varphi^*(r\cos\theta,r\sin\theta,t)
=\min_{s\ge0} h(s,\theta),
\]

where
\[
h(s,\theta):=\frac{(-r\cos\theta+b s+c t)^2}{4(t+a s)}+\frac{(r\sin\theta+s)^2}{4t},\qquad s\ge0,\ \theta\in\big[0,\tfrac{\pi}{2}\big].
\]

We first show that for each fixed \(s\ge0\) the function \(\theta\mapsto h(s,\theta)\) is nondecreasing on \([0,\tfrac{\pi}{2}]\). Compute the partial derivatives (with \(s\) fixed):
\[
\frac{\partial h}{\partial x}=\frac{x-b s-c t}{2(t+a s)},\qquad
\frac{\partial h}{\partial y}=\frac{y+s}{2t}.
\]

Using \(x=r\cos\theta\) and \(y=r\sin\theta\), the chain rule yields
\[
\frac{\partial h}{\partial\theta}
= -r\sin\theta\cdot\frac{r\cos\theta-b s-c t}{2(t+a s)}
  + r\cos\theta\cdot\frac{r\sin\theta+s}{2t}.
\]

Multiply both sides by \(\dfrac{2t(t+a s)}{r}\) (which is positive), to obtain
\[
\frac{2t(t+a s)}{r}\,\frac{\partial h}{\partial\theta}
= -t(r\cos\theta-b s-c t)\sin\theta + (t+a s)(r\sin\theta+s)\cos\theta.
\]

Expanding and simplifying gives
\[
\frac{2t(t+a s)}{r}\,\frac{\partial h}{\partial\theta}
= t(b s + c t)\sin\theta + a s\, r\sin\theta\cos\theta + (t+a s)s\cos\theta.
\]

Under the hypotheses \(a,b,c > 0\), \(t>0\), \(s\ge0\), \(r\ge0\) and \(\theta\in[0,\tfrac{\pi}{2}]\), the right-hand side is a sum of nonnegative terms; therefore \(\partial_\theta\,h(s,\theta)\ge0\). Hence for each \(s\ge0\) the map \(\theta\mapsto h(s,\theta)\) is nondecreasing.

\medskip

Finally, given \(0\le\theta_1\le\theta_2\le\tfrac{\pi}{2}\), for every \(s\ge0\) we have \(h(s,\theta_1)\le h(s,\theta_2)\). Taking the minimum over \(s\ge0\) on both sides yields \(\Phi(\theta_1)\le\Phi(\theta_2)\). Thus \(\Phi\) is nondecreasing on \([0,\tfrac{\pi}{2}]\).

\medskip

Since \(J(\mathbf{x},t)=\varphi^*(\mathbf{x},t)-t\), subtracting the constant \(t\) does not affect monotonicity, and the proof is complete.

\end{proof}

\section*{ACKNOWLEDGMENTS}
The author gratefully acknowledges Professor Haomin Huang for his meticulous feedback and insightful suggestions, which significantly improved the clarity and rigor of the manuscript. The discussions with him during the completion of the proofs were instrumental in strengthening the robustness of the theoretical framework.

\end{document}